\documentclass{amsart}

\usepackage{amsfonts,amssymb}
\usepackage{multicol}
\usepackage{pstricks}
\usepackage{amsmath}
\usepackage{latexsym}
\usepackage{mathrsfs}
\usepackage{epsfig}

\input{boxedeps}
\SetEPSFDirectory{} 
\HideDisplacementBoxes

\SetRokickiEPSFSpecial  

\newtheorem{theorem}{Theorem}[section]
\newtheorem{lemma}[theorem]{Lemma}
\newtheorem{corollary}[theorem]{Corollary}

\theoremstyle{definition}

\newtheorem{problem}[theorem]{Problem}

\theoremstyle{remark}
\newtheorem{remark}[theorem]{Remark}
\newtheorem{siegel}{The $n$-dimensional manifold Siegel problem\hspace{-.045in}}

\numberwithin{equation}{section}

\newcommand{\integers}{{\mathbb Z}}
\newcommand{\realnos}{{\mathbb R}}
\newcommand{\rationalnos}{{\mathbb Q}}

\def\diag{{\rm diag}}

\def\Eta{{\mathrm H}}
\def\Kappa{{\mathrm K}}

\def\R{\realnos}
\def\Z{\integers}
\def\Q{\rationalnos}

\def\o{\text{O}}

\def\po{\text{PO}}

\def\vol{\text{vol}}
\def\covol{\text{covol}}
\def\isom{\text{Isom}\,}
\newpsobject{showgrid}{psgrid}{subgriddiv=1,griddots=5,gridlabels=6pt}
\font\tencyr=wncyr10 at 10truept
\input cyracc.def
\def\cyr{\tencyr\cyracc}
\begin{document}

\title[Hyperbolic manifolds and a problem of Siegel]{Right-angled Coxeter polytopes, 
hyperbolic 6-manifolds, and a problem of Siegel}

\author{Brent Everitt}
\address{Department of Mathematics, University of York, York YO10 5DD, England}
\email{bje1@york.ac.uk}
\thanks{The first author thanks the Department of Mathematics at the University of
Vanderbilt for its hospitality during a visit in 2010}

\author{John G. Ratcliffe}
\address{Department of Mathematics, Vanderbilt University, Nashville, TN 37240, USA}
\email{j.g.ratcliffe@vanderbilt.edu}
\thanks{The second author thanks the London Mathematical Society for support to speak at the Universities of Durham, Warwick, and York in 2009}

\author{Steven T. Tschantz}
\address{Department of Mathematics, Vanderbilt University, Nashville, TN 37240, USA}
\email{steven.tschantz@vanderbilt.edu}
\thanks{}





\begin{abstract}
By gluing together the sides of eight copies of an all-right angled hyperbolic 6-dimensional polytope,  
two orientable hyperbolic $6$-manifolds
with Euler characteristic $-1$ are constructed. They are the first known
examples of orientable hyperbolic $6$-manifolds having the smallest possible volume. 
\end{abstract}

\maketitle

\section{Introduction} \label{section:1}
In 1931 L\"obell glued together the sides of eight copies of an all-right angled 
hyperbolic polyhedron to create the first example of a compact, orientable, hyperbolic $3$-manifold
\cite{Lobell}.
In the subsequent decades, the universe of hyperbolic manifolds has grown into
an extraordinarily rich one. To get a handle on all this richness, one must resort to
viewing it through the eyes of invariants. For hyperbolic 
manifolds the most important invariant is \emph{volume\/}.

Viewed this way the $n$-dimensional hyperbolic manifolds look like a
well-ordered set, which is discrete, when $n\not= 3$, and we are naturally led to

\begin{siegel}
Determine the minimum possible volume obtained 
by an orientable hyperbolic $n$-manifold. 
\end{siegel}

All
hyperbolic manifolds in this paper will be complete Riemannian manifolds of constant 
sectional curvature $-1$. 
The ``full'' Siegel problem refers to the problem above for orbifolds rather than manifolds.
Nevertheless, the manifold Siegel problem is one with a long and venerable history.

This paper describes our solution when $n=6$. But first, an overview and some background.
The Euler characteristic $\chi$ creates a big difference between even and odd dimensions.
When $n$ is even, the Gauss-Bonnet theorem gives
$$
\vol(M)=\kappa_n\chi(M),\text{ with }
\kappa_n=(-2\pi)^{\frac{n}{2}} /(n-1)!!
$$
for the volume of an $n$-dimensional hyperbolic manifold $M$.
As $\chi(M)\in\Z$, the most obvious place to look for solutions to the problem is when
$|\chi|=1$. A compact orientable
$M$ satisfies $|\chi(M)|\in 2\Z$, so the minimum volume is most
likely achieved by a non-compact manifold. 
When $n$ is odd, $\chi(M)=0$, and a different
approach must be found. For these reasons progress in even dimensions has been more rapid. 
In practice the basic problem stated above tends to 
bifurcate into a number of sub-problems by the addition of
more adjectives: the most common being compact, non-compact, and arithmetic.

In dimension $2$, the answer is $2\pi$, achieved by the once-punctured torus or the
thrice-punctured sphere, and $4\pi$ in the compact case--achieved, of course, by the genus two surface.
Dimension $3$ is the only odd one for which the problem is solved: the minimum volume,
realized by the Matveev-Fomenko-Weeks manifold \cite{G-M-M}, is
$$
\frac{3\cdot 23^{\frac{3}{2}}\,\zeta_k(2)}{4\pi^4},
$$
with $\zeta_k$ the Dedekind zeta function of the number field $k=\Q(\theta)$
where $\theta$ satisfies $\theta^3-\theta+1=0$ (see \cite{Chinburg01} for this 
beautiful formula for its volume).
The non-compact solution is over twice as large: $6\text{{\cyr L}}(\pi/3)$, with {\cyr L} the
Lobachevsky function, and is achieved by the figure-eight knot complement \cite{Meyerhoff01}. 

In dimension $4$, the minimum of $4\pi^2/3$ is realized in \cite{Ratcliffe00} by gluing the sides
of an ideal, regular, hyperbolic $24$-cell (see also \cite{Everitt02} for an algebro-combinatorial 
construction). For compact manifolds, the answer lies somewhere in the range 
$8\pi^2/3\leq\vol\leq 64\pi^2/3$, with the former the theoretical $\chi\geq 2$ bound, and the
latter realized in \cite{Conder04,Long} as orientable double covers of non-orientable
manifolds with $\chi=8$. Very little is known about the general picture in dimension $5$. 
The smallest currently known volume is $7\zeta(3)/4$, where $\zeta$ is the Riemann 
zeta function--see \cite{Ratcliffe04}.

In \cite{E-R-T} we announced the discovery of some hyperbolic $6$-manifolds with minimum 
volume $8\pi^3/15$, although they were non-orientable. 
In this paper we construct, both geometrically
and algebraically, \emph{orientable\/} examples having this minimal volume. 
The manifold Siegel problem thus has answer $8\pi^3/15$ in dimension $6$. 
For compact 6-manifolds nothing is known except for the bound $\vol\geq 16\pi^3/15$ 
given by the Gauss-Bonnet theorem. 

Our construction is a return to the classical methods of L\"obell: the manifolds are
obtained by gluing together the sides of 8 copies of an all right-angled, hyperbolic, 6-dimensional polytope that we call $P^6$. 

Indeed the construction would not have been possible 
without the exquisite, extraordinary, and exceptional polytope $P^6$.
The polytope is exquisite because 
of its high degree of symmetry--its symmetry group is the Weyl group of type 
$E_6$ of order $51,840$;
extraordinary, because $P^6$ is unique up to congruence 
with the property that it has the minimum number of $k$-dimensional faces  
for each $k = 1,\ldots, 5$ among all right-angled, hyperbolic, 6-dimensional polytopes (see
\S\ref{section:3});  
exceptional, because the Euclidean dual
is a $6$-dimensional semi-regular polytope, first discovered by Gosset 
in 1900, which combinatorially parametrizes the celebrated 
arrangement of $27$ straight lines on a general cubic surface. 
Another amazing property of $P^6$ is that it is a Coxeter polytope 
for the congruence two subgroup of the group $PO_{6,1}\Z$
of integral, positive, Lorentzian $7\times 7$ matrices (see \S\ref{section:2}).
This breathtaking confluence of hyperbolic geometry, 
Coxeter groups, and number theory gave us 
the necessary insight to find the right gluings of the 8 copies of $P^6$ needed to construct the manifolds.

The paper is organized as follows: \S\ref{section:2} introduces two families that turn out to be 
hyperbolic Coxeter groups for small $n$--the groups $PO_{n,1}\Z$ and their congruence two subgroups. 
We also introduce the family of hyperbolic polytopes $P^n$. 
These guys really are the stars of the show, and we continue with some of
their remarkable properties in \S\S\ref{section:3}-\ref{section:4}. 
The geometric construction of the manifolds is described in \S\ref{section:5}. 
We also remind the reader of the non-orientable results obtained in \cite{E-R-T}, 
and in \S\ref{section:6} of the coding conventions used. 
In  \S\S\ref{section:7}-\ref{section:8} we give an alternative construction for those 
who would rather check an algebraic argument by hand than trust a computer calculation. 
Finally, in \S\ref{section:9} we finish with some general remarks and open problems.

\section{Two families of hyperbolic Coxeter polytopes} \label{section:2}


A {\it polytope} $P$ in $n$-dimensional hyperbolic space
$H^n$ is a convex polyhedron, with finitely many actual 
and ideal vertices,
such that $P$ is the convex hull of all its vertices in the projective disk model of $H^n$. 
By \cite[Theorem 6.5.8]{Ratcliffe06}, an $n$-dimensional polyhedron in $H^n$ is a polytope 
if and only if it has finitely many sides (that is, $(n-1)$-dimensional faces) and finite volume. 
Call $P$ a \emph{Coxeter polytope\/} if the 
dihedral angle subtended by two adjacent sides 
is $0$ or $\pi/m$ for some integer $m\geq 2$ (a polytope has a dihedral angle $0$
only if it is $2$-dimensional); $P$ is 
\emph{right-angled\/} if all its dihedral angles are either $0$ or $\pi/2$. 

The group $\Gamma$ generated by reflections in the sides of the Coxeter polytope 
$P$ is a hyperbolic Coxeter group and a discrete subgroup of $\isom H^n\cong PO_{n,1}\R$.
The Coxeter symbol for $P$ (or $\Gamma$) has nodes
indexed by the sides, and an edge labeled $m$ joining the nodes
corresponding to adjacent sides that intersect with angle $\pi/m$. 
The edges joining the nodes corresponding to non-adjacent sides are labeled $\infty$. 
The labels $2$ and $3$ occur often in real world examples, 
so such edges are respectively removed or left unlabeled.

Finding hyperbolic Coxeter groups is not entirely straightforward, especially in higher 
dimensions.
A fruitful source is via the automorphism groups 
of certain integral quadratic forms. For example, 
consider a $(n+1)$-dimensional Lorentzian lattice, 
that is, an $(n+1)$-dimensional 
free $\Z$-module equipped with a $\Z$-valued bilinear form of signature $(n,1)$. 
For each $n$ there is a unique such, denoted
$I_{n,1}$, that is odd and self-dual (see \cite[Theorem V.6]{Serre73}, or \cite{Milnor73, Neumaier83}).
By \cite{Borel62} the group $\text{Aut}(I_{n,1})\cong O_{n,1}\Z$ (hence also $PO_{n,1}\Z$)  acts 
discretely, cofinitely by isometries on $I_{n,1}\otimes\R\approx H^n$. 

The study of the action of $PO_{n,1}\Z$ on $H^n$ originated with Fricke \cite{Fricke} 
who showed that $PO_{2,1}\Z$ is a hyperbolic Coxeter group with symbol in Figure \ref{fig:1}.  
Coxeter and Whitrow \cite{Coxeter50} showed that  $PO_{3,1}\Z$ is 
also a hyperbolic Coxeter group with symbol in Figure \ref{fig:1}.  
In general, we have a split, short, exact sequence
\begin{equation}
  \label{eq:1}
1\rightarrow \Gamma^n\rightarrow PO_{n,1}\Z \rightarrow H\rightarrow 1,  
\end{equation}
where $\Gamma^n$ is the subgroup generated by the reflections in $PO_{n,1}\Z$
and $H$ is the symmetry group of the Coxeter polytope (or symbol) corresponding to $\Gamma^n$. 
Vinberg and Kaplinskaja \cite{Vinberg72,Vinberg78} showed that $H$ is a finite group (and hence
$\Gamma^n$ acts with finite covolume)  if and only if
$n\leq 19$, and they computed the Coxeter symbols for $\Gamma^n$ in these cases
(see \cite[Table 4]{Vinberg78a}).

We now restrict to the cases $2\leq n\leq 8$, where the Coxeter polytope for $\Gamma^n$
is a finite volume non-compact simplex $\Delta^n$ with Coxeter symbol in 
Figure \ref{fig:1}. Each symbol has $n+1$ nodes and $\Delta^n$ has $n$ actual vertices (that
is, vertices in $H^n$) and a single ideal vertex on $\partial H^n$
opposite the side labeled $1$ in the Coxeter symbol.
None of these symbols have any symmetries, so the group $H$ in (\ref{eq:1}) is trivial
and $PO_{n,1}\Z=\Gamma^n$ is itself a reflection group. Indeed, this is true also for $9\leq n\leq 13$,
as can be seen by inspecting the symbols in \cite[Table 4]{Vinberg78a}, but this need not concern 
us here. 

From now on we use the hyperboloid model 
$$H^n=\{x\in \R^{n,1}: \langle x,x \rangle=-1\text{ and }\ x_{n+1}>0\},$$
where $\langle\_\, ,\_\, \rangle$ is the 
Lorentzian inner product on $\realnos^{n,1}$ defined by \cite[Formula 3.1.2]{Ratcliffe06}. 
The vertices of $\Delta^n$ in $\realnos^{n,1}$ are listed in Table \ref{table:1} 
with the $i$th vertex opposite the side labelled $i$ in Figure 1. 

\begin{table}[b]  
$$\begin{array}{llll}
2 & (1,0,1) & (1,1,2)/\sqrt{2} & (0,0,1) \\
3 & (1,0,0,1) &  (1,1,0,2)/\sqrt{2} &  (1,1,1,3)/\sqrt{6} \\ 
          &  (0,0,0,1) &                               &       \\
4 & (1,0,0,0,1) & (1,1,0,0,2)/\sqrt{2} & (1,1,1,0,3)/\sqrt{6} \\
          & (1,1,1,1,3)/\sqrt{5} &  (0,0,0,0,1) &      \\
5  & (1,0,0,0,0,1)  &  (1,1,0,0,0,2)/\sqrt{2} & (1,1,1,0,0,3)/\sqrt{6} \\
          & (1,1,1,1,0,3)/\sqrt{5} &    (1,1,1,1,1,3)/2 &  (0,0,0,0,0,1) \\
6 & (1,0,0,0,0,0,1) &  (1,1,0,0,0,0,2)/\sqrt{2} & (1,1,1,0,0,0,3)/\sqrt{6} \\
          & (1,1,1,1,0,0,3)/\sqrt{5} &  (1,1,1,1,1,0,3)/2 & (1,1,1,1,1,1,3)/\sqrt{3} \\
          & (0,0,0,0,0, 0,1) &           &          \\
7 & (1,0,0,0,0,0,0,1)  & (1,1,0,0,0,0,0,2)/\sqrt{2}  &  (1,1,1,0,0,0,0,3)/\sqrt{6} \\
          & (1,1,1,1,0,0,0,3)/\sqrt{5}  &  (1,1,1,1,1,0,0,3)/2   & (1,1,1,1,1,1,0,3)/\sqrt{3}  \\
          &  (1,1,1,1,1,1,1,3)/\sqrt{2} &   (0,0,0,0,0,0,0,1) &   \\
8 & (1,0,0,0,0,0,0,0,1)  & (1,1,0,0,0,0,0,0,2)/\sqrt{2}  &   (1,1,1,0,0,0,0,0,3)/\sqrt{6} \\
          & (1,1,1,1,0,0,0,0,3)/\sqrt{5} &   (1,1,1,1,1,0,0,0,3)/2   &  (1,1,1,1,1,1,0,0,3)/\sqrt{3}  \\
          & (1,1,1,1,1,1,1,0,3)/\sqrt{2}  &  (1,1,1,1,1,1,1,1,3)  &  (0,0,0,0,0,0,0,0,1)
\end{array}$$
\caption{The vertices of $\Delta^n$ listed so that the $i$th vertex is 
opposite the $i$th side in Figure 1. The first column lists $n$. }
\label{table:1}
\end{table}
\begin{figure}[tb]
\centering
\begin{pspicture}(0,0)(12.5,5)
\rput(6.25,1.2){\BoxedEPSF{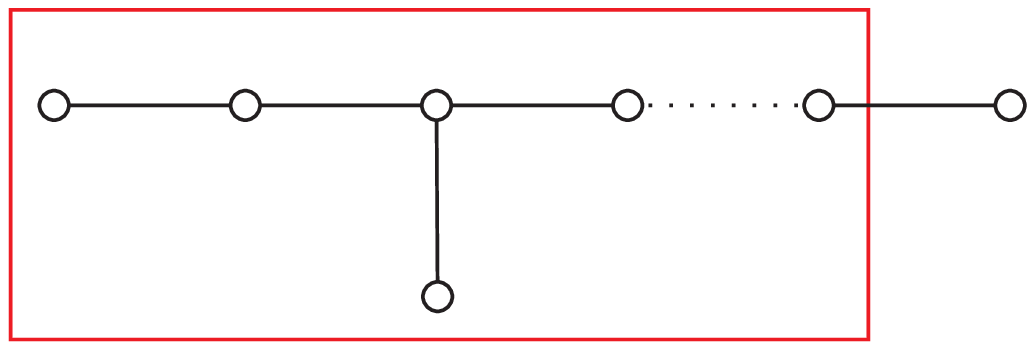 scaled 750}}
\rput(9,3.75){\BoxedEPSF{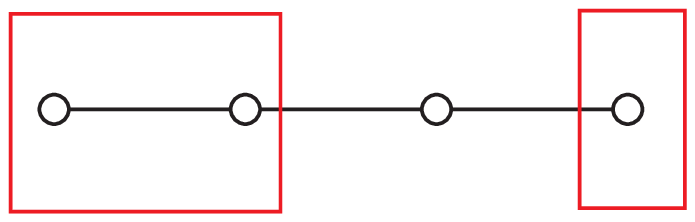 scaled 750}}
\rput(3,3.75){\BoxedEPSF{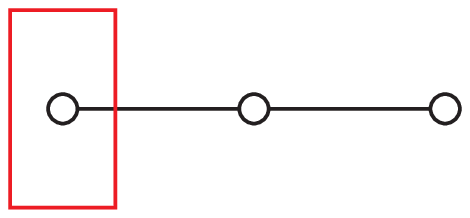 scaled 750}}
\rput(0,-.05){\rput(2.7,2.1){${\scriptstyle 1}$}\rput(4.2,2.1){${\scriptstyle 2}$}
\rput(5.65,2.1){${\scriptstyle 3}$}\rput(7.05,2.1){${\scriptstyle 4}$}
\rput(8.5,2.1){${\scriptstyle n-1}$}\rput(10,2.1){${\scriptstyle n}$}
\rput(5.2,.3){${\scriptstyle n+1}$}}
\rput(-1.05,1.35){\rput(2.7,2.1){${\scriptstyle 1}$}\rput(4.2,2.1){${\scriptstyle 2}$}
\rput(5.65,2.1){${\scriptstyle 3}$}}
\rput(4.05,1.35){\rput(2.7,2.1){${\scriptstyle 1}$}
\rput(4.2,2.1){${\scriptstyle 2}$}\rput(5.65,2.1){${\scriptstyle 3}$}
\rput(7.05,2.1){${\scriptstyle 4}$}}
\rput(3.9,4){$\infty$}\rput(2.4,4){$4$}
\rput(8.95,4){$4$}\rput(10.4,4){$4$}
\rput(9.3,2){$4$}
\rput(3,4.5){$n=2$}\rput(9.5,4.5){$n=3$}\rput(10.5,1){$4\leq n\leq 8$}
\rput(2.4,3.2){${\red \Sigma^2}$}\rput(9.7,3.1){${\red \Sigma^3}$}
\rput(8,.75){${\red \Sigma^n}$}
\psline[linewidth=.2mm,linecolor=red]{->}(9.3,3.1)(8.45,3.1)
\psline[linewidth=.2mm,linecolor=red]{->}(10,3.1)(10.75,3.1)
\end{pspicture}
\caption{Coxeter symbols for the groups $PO_{n,1}\Z$ for $2\leq n\leq 8$.}
\label{fig:1}
\end{figure}

Let $\Sigma^n\subset\Gamma^n$ be the subgroup generated by the reflections
in the sides  indicated in the symbol in Figure \ref{fig:1}. 
The $\Sigma^n$ are isomorphic to Weyl groups of types
$A_1,A_1\times A_2, A_4, D_5, E_6, E_7$, and $E_8$, and hence are finite with orders
$2,12,120,1920,51840,$ $2903040$, and $696729600$, respectively.  
For $3\leq n\leq 8$, there is a single $\Gamma^n$-generating reflection $s$ not in $\Sigma^n$. 
Let $\Sigma_0^n\subset\Sigma^n$
be the subgroup generated by those $\Gamma^n$-generating reflections commuting with $s$. 
Then $\Sigma_0^n$ has symbol with those nodes in the symbol for $\Sigma^n$ not joined to $s$. 
The $\Sigma_0^n$ are also Weyl groups of types $A_1,A_1\times A_2, A_4, D_5, E_6$, and $E_7$.

We represent the isometries of $H^n$ by positive Lorentzian $(n+1)\times (n+1)$ matrices,
{\it positive} here meaning that they preserve the sign of the last or ``time'' coordinate. 
The generators of $\Gamma^n$ are then \emph{integral\/} positive Lorentzian matrices.
Letting $\mathrm{Perm}(i,j)$ be the permutation $(n+1)\times(n+1)$ matrix corresponding to 
the transposition $(i,j)$, then for $2\leq n\leq 8$, the reflection $s_i$ in the side of $\Delta^n$ 
labeled $i$
is represented by $\text{Perm}(i,i+1)$, for $1\leq i\leq n-1$;
the reflection $s_n$ in the side labeled $n$ is given by the diagonal matrix $D =\mathrm{diag}(1,\ldots,1,-1,1)$, 
and $s_{n+1}$ in the $(n+1)$st side by the following $(n+1)\times (n+1)$ matrix:
$$
\begin{pspicture}(0,0)(12.5,3.25)
\rput(1.5,1.625){
$\left(\begin{array}{rrr}
-1 & -2 & 2 \\
-2 & -1 & 2 \\
-2 & -2 & 3
\end{array}\right)$
}
\rput(5,1.625){
$\left(\begin{array}{rrrr}
0 & -1 & -1& 1 \\
-1 & 0 & -1 & 1\\
-1 & -1 & 0 & 1 \\
-1 & -1 & -1 & 2
\end{array}\right)$
}
\rput(7.2,0){
\rput[lb](0,0){$\begin{array}{rrrrcrr}
0 & -1 & -1 & & & & 1 \\
-1& 0  & -1 & & & & 1 \\
-1&-1  & 0  & & & & 1 \\
& & &1 & & & \\
& & & & \ddots & & \\
& & & & & 1 & \\
-1 & -1 & -1 & & & & 2\\
\end{array}$}
\psline[linewidth=.2mm]{-}(.1,1.95)(4.7,1.95)
\psline[linewidth=.2mm]{-}(2.4,0)(2.4,3.2)
\rput[lb](-.2,-.1){$\left(\begin{array}{c}
\vrule width 0 mm height 32 mm depth 0 pt\\
\end{array}\right.$}
\rput[lb](4.4,-.1){$\left.\begin{array}{c}
\vrule width 0 mm height 32 mm depth 0 pt\\
\end{array}\right).$}
\rput(1.2,1.1){$\text{\Large{0}}$}\rput(3.5,2.6){$\text{\Large{0}}$}
\rput(3.2,.5){$\text{\Large{0}}$}\rput(4.1,1.4){$\text{\Large{0}}$}
}
\rput(1.5,2.6){$n=2:$}\rput(5,2.75){$n=3:$}\rput(6.2,.2){$4\leq n\leq 8:$}
\end{pspicture}
$$

Let $\Gamma_2^n$ be the congruence two subgroup 
$\Gamma^n\cap\ker(PO_{n,1}\Z\rightarrow PO_{n,1}\Z/2)$, 
that is, the group of all matrices in $\Gamma^n$ that are 
congruent to the identity matrix modulo two. By \cite[Lemma 16]{Ratcliffe97}
the index of $\Gamma_2^n$ in $\Gamma^n$ is given by
$$
[\Gamma^n:\Gamma^n_2]=
\frac{{\displaystyle  2\cdot(2^2-1)\cdot 2^3\cdot(2^4-1)\cdots(2^n-\varepsilon_2(n))}}{{\displaystyle 2^{(n-1)}+2^{(n-1)/2}\cos((n-1)\pi/4)}},
$$
where $\varepsilon_2(n) = 0$ if $n$ is odd and $\varepsilon_2(n) = 1$ if $n$ is even.  
In particular, $[\Gamma^n:\Gamma^n_2]=|\Sigma^n|$
when $2\leq n\leq 7$. For $n=8$, we have $|\Sigma^8|=2[\Gamma^8:\Gamma^8_2]$.

The kernel of the map $\Gamma^n\rightarrow\Sigma^n$ that kills the 
$\Gamma^n$-generators not in $\Sigma^n$ and is the
identity on the others is contained in $\Gamma^n_2$ since the generators killed
are represented by matrices in $\Gamma^n_2$. As this kernel and $\Gamma_2^n$ have the same index, they are equal, 
and we have a split, short, exact sequence
\begin{equation}
  \label{eq:2}
1\rightarrow \Gamma_2^n\rightarrow \Gamma^n\rightarrow \Sigma^n\rightarrow 1,  
\end{equation}
for $2\leq n\leq 7$ (for $n=8$ the kernel is a subgroup of index $2$ in $\Gamma_2^n$ 
which we will describe below).

Remarkably, it turns out that when $2\leq n\leq 7$ the congruence subgroup $\Gamma_2^n$ is also
a reflection group corresponding to a very nice Coxeter polytope.
This polytope is essentially the smallest right-angled Coxeter polytope  that can be obtained
by gluing together copies of the the simplex $\Delta^n$. 

\begin{figure}
\centering
\begin{pspicture}(0,0)(12.5,6)
\rput(6.25,2.25){\BoxedEPSF{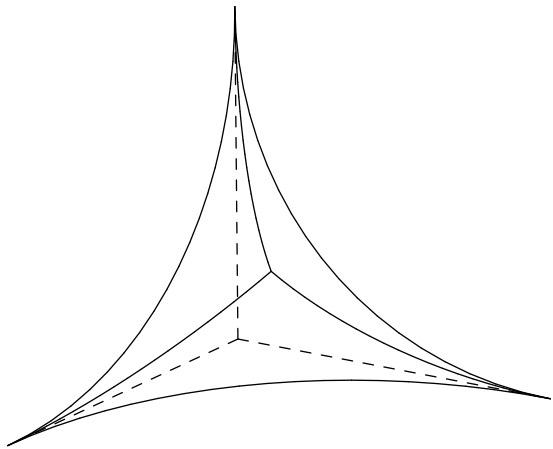 scaled 1000}}
\end{pspicture}
\caption{The right-angled Coxeter polytope $P^3$ in the conformal ball model of $H^3$.}
\label{fig:2}
\end{figure}

Specifically, for $2\leq n \leq 8$, let
$$
P^n=\Sigma^n\Delta^n:=\bigcup_{\gamma\in\Sigma^n} \gamma(\Delta^n),
$$
a polytope obtained by gluing $|\Sigma^n|$ copies of the simplex $\Delta^n$ together. 
Thus $P^n$ has finite volume and like $\Delta^n$ is non-compact, 
with a mixture of actual vertices in $H^n$ and ideal vertices on $\partial H^n$. 

The triangle $\Delta^2$ has angles $0, \pi/4, \pi/2$. 
The polygon $P^2$ is the union of just two copies of $\Delta^2$ obtained by
reflecting across the finite side of $\Delta^2$ labeled $1$, and so $\Sigma^2$ 
is the group of symmetries of $P^2$. 
The polygon $P^2$ is a right triangle with two ideal vertices. 

Let $v$ be the vertex of $\Delta^n$ opposite the side labeled $n$ when $3\leq n\leq 8$. 
Then $\Sigma^n$ is the stabilizer in $\Gamma^n$ of $v$
and $P^n$ has sides the $\Sigma^n$-images of the side of $\Delta^n$ opposite $v$. 
The group of symmetries of $P^n$ is thus $\Sigma^n$, with $\Delta^n$ a fundamental polytope 
for the action of $\Sigma^n$ on $P^n$. 

For $3\leq n\leq 8$,  the polytope $P^n$ is an all right-angled Coxeter polytope 
with $|\Sigma^n|/|\Sigma_0^n|$ sides each congruent to $P^{n-1}$. 
This follows as the side of $\Delta^n$ opposite $v$ is congruent to $\Delta^{n-1}$ 
and intersects the other sides of $\Delta^n$ in dihedral angles $\pi/2$ or $\pi/4$. 
Figure \ref{fig:2} shows $P^3$. 
The polytope $P^n$ has $3, 6, 10, 16, 27, 56$, and $240$ sides for $n = 2,\ldots, 8$, respectively. 

By (\ref{eq:2}), $\Sigma^n$ is a set of coset representatives for 
$\Gamma^n_2$ in $\Gamma^n$, and so $P^n$ is a fundamental region for the action
of $\Gamma^n_2$ on $H^n$ for $2\leq n\leq 7$. 
Moreover, the reflections in the sides of $P^n$ are just the 
$\Sigma^n$-conjugates of the $\Gamma^n$-generators not in $\Sigma^n$. 
Since these generators lie in $\Gamma_2^n$,
the reflections in the sides of $P^n$ lie in (the normal subgroup) $\Gamma^n_2$ also. Thus the reflection
group generated by reflections in the sides of $P^n$ is a subgroup of $\Gamma^n_2$ 
but having the same fundamental region as $\Gamma^n_2$. The conclusion is that the reflection
group and $\Gamma^n_2$ are equal, and we have proved:

\begin{theorem}[see also \cite{Ratcliffe00,Ratcliffe04}]\label{thm:1} 
For $2\leq n\leq 7$, the congruence two subgroup 
$\Gamma_2^n$ of the group 
$\Gamma^n=PO_{n,1}\Z$ of integral, positive, 
Lorentzian $(n+1)\times(n+1)$ matrices is a hyperbolic Coxeter group with 
Coxeter polytope the right-angled polytope $P^n$. 
\end{theorem}

The group $\Gamma^8_2$ is almost a Coxeter group in the sense that
the hyperbolic Coxeter group $\Gamma(P^8)$ generated by the reflections in the sides of $P^8$ 
is a subgroup of $\Gamma^8_2$ of index 2.   Moreover, we have split, short, exact sequences 
$$1\rightarrow \Gamma(P^8) \rightarrow \Gamma^8\rightarrow \Sigma^8\rightarrow 1, $$ 
$$1\rightarrow \Gamma(P^8) \rightarrow \Gamma^8_2\rightarrow \Gamma^8_2\cap \Sigma^8\rightarrow 1.$$ 
The group $\Gamma^8_2\cap \Sigma^8$ is generated by the longest element of $\Sigma^8$ and
represented by the Lorentzian $9\times 9$ matrix 
$$
\left(\begin{array}{rrrrrrrrr}
-3 & -2  & -2  & -2 & -2  & -2 & -2 & -2 & 6 \\
-2 & -3  & -2  & -2 & -2  & -2 & -2 & -2 & 6 \\
-2 & -2  & -3  & -2 & -2  & -2 & -2 & -2 & 6 \\     
-2 & -2  & -2  & -3 & -2  & -2 & -2 & -2 & 6 \\  
-2 & -2  & -2  & -2 & -3  & -2 & -2 & -2 & 6 \\  
-2 & -2  & -2  & -2 & -2  & -3 & -2 & -2 & 6 \\
-2 & -2  & -2  & -2 & -2  & -2 & -3 & -2 & 6 \\ 
-2 & -2  & -2  & -2 & -2  & -2 & -2 & -3 & 6 \\       
-6 & -6  & -6  & -6 & -6  & -6 & -6 & -6 & 17 \\
\end{array}\right).
$$

Theorem \ref{thm:1} says that $\Gamma^n_2$ is a right-angled Coxeter group for $2\leq n \leq 7$. 
An interesting Coxeter group ought to have an interesting Coxeter 
symbol--and so it does, at least for $4\leq n\leq 7$. 
The Euclidean duals of the polytopes $P^n$ for $4 \leq n \leq 8$ are semi-regular polytopes, 
which were first discovered by T. Gosset \cite{Gosset} 
and called ``pure Archimedean'' by Coxeter in his first paper \cite{Coxeter28}. 
In modern notation, these semi-regular polytopes are denoted $k_{21}$ for $k= 0,\ldots, 4$. 
In particular, the Coxeter symbol for $\Gamma^n_2$ can be obtained from the $1$-skeleton
of the corresponding Gosset polytope: take the complement of the edges 
in the $1$-skeleton by removing them 
if they exist and inserting them if they don't, and label all the edges that result by $\infty$.
See Figure \ref{fig:3} for the $4$-dimensional case (where we've left off the $\infty$ labels for clarity).

\begin{figure}
\centering
\begin{pspicture}(0,0)(12.5,5)
\rput(3,2.5){\BoxedEPSF{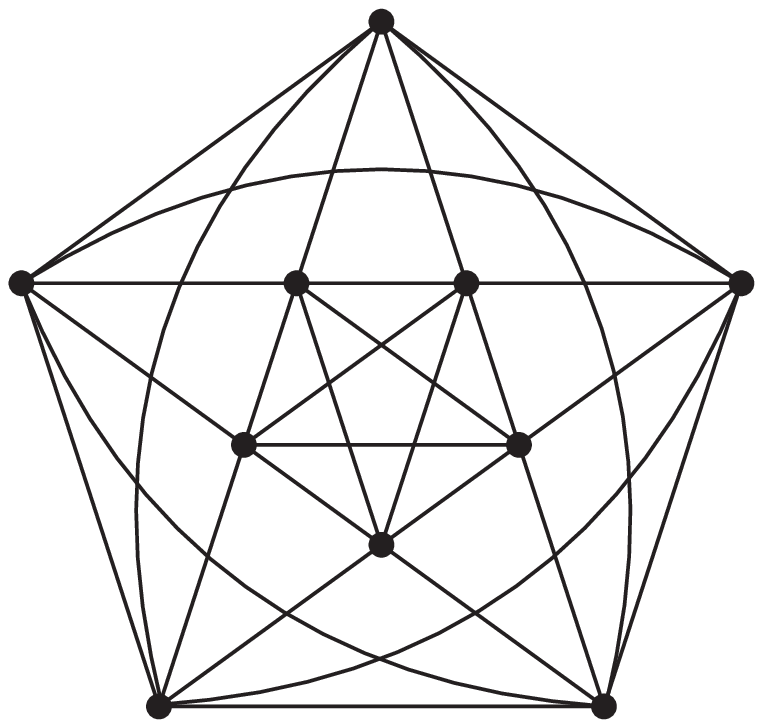 scaled 650}}
\rput(9,2.5){\BoxedEPSF{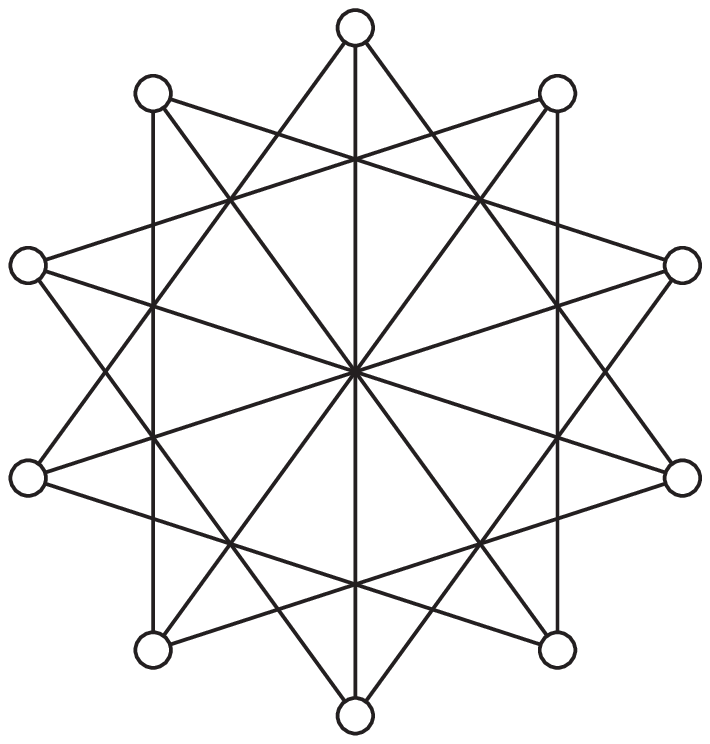 scaled 650}}
\end{pspicture}
\caption{$1$-skeleton of the $4$-dimensional Gosset polytope $0_2$ (left)
and Coxeter symbol for the congruence subgroup $\Gamma^4_2$ (right). The 
$1$-skeleton is combinatorially but not geometrically accurate and the
$\infty$ labels have been left off the Coxeter symbol for clarity.}
\label{fig:3}
\end{figure}

The torsion in a Coxeter group is well understood (see for instance 
\cite{Everitt02,E-H} and the references there), 
so as a corollary we will describe 
the torsion of $\Gamma_2^n$ for $2\leq n \leq 7$.  
We will need a precise description of the torsion of $\Gamma^6_2$ in \S 7. 

Let $\Kappa^n$ be the group of $(n+1)\times (n+1)$ matrices of the form 
$\diag(\pm 1,\ldots, \pm 1, 1)$.  Then $\Kappa^n$ is an elementary 2-group of order $2^n$, 
and $\Kappa^n$ is the stabilizer in $\Gamma_2^n$ of the actual vertex $e_{n+1}$ of $P^n$ for $2 \leq n \leq 7$. 
Let $\ell_n$ be the line in $H^n$ whose ideal endpoints are represented by 
the light-like vectors $e_1+e_{n+1}$ and $e_2+e_{n+1}$ in $\realnos^{n,1}$ for $2\leq n \leq 7$. 
Then $\ell_n$ is a line edge of $P^n$, and $\ell_n$ is contained in $n-1$ sides of $P^n$ for $2\leq n \leq 7$. 
Let $\Lambda^n$ be the group generated by the reflections in the sides of $P^n$ that contain $\ell_n$ 
for $2\leq n \leq 7$. Then $\Lambda^n$ is an elementary 2-group of order $2^{n-1}$, 
and $\Lambda^n$ is the pointwise stabilizer of $\ell_n$ for $2\leq n \leq 7$. 

Let $a_n$ be the number of actual vertices of $P^n$ for $n=2,\ldots, 7$. 
Then $a_n = 1, 2, 5,16, 72, 576$, respectively. 
Let $l_n$ be the number of line edges of $P^n$ for $n=2,\ldots, 7$. 
Then $l_n = 1, 3,10, 40, 216, 2016$, respectively.

\begin{corollary}\label{corollary:1}
For $2\leq n \leq 7$, every finite subgroup of $\Gamma_2^n$ is conjugate in $\Gamma^n$ to a subgroup of either the elementary 2-group $\Kappa^n$ or the elementary 2-group $\Lambda^n$, 
and there are $a_n$ conjugacy classes of maximal finite subgroups 
of $\Gamma_2^n$ of order $2^n$, represented by the stabilizers in $\Gamma_2^n$ of the actual vertices of $P^n$, 
and there are $l_n$ conjugacy classes of maximal finite subgroups of $\Gamma_2^n$ of order $2^{n-1}$, 
represented by the pointwise stabilizers in $\Gamma_2^n$ of the line edges of $P^n$. 
\end{corollary}
\begin{proof}
By \cite[Exercise 7.1.7]{Ratcliffe06}, every finite subgroup of $\Gamma_2^n$ 
is conjugate in $\Gamma_2^n$ 
to a subgroup of either a stabilizer of an actual vertex of $P^n$ or a pointwise 
stabilizer of a line edge of $P^n$. 
The group $\Gamma^n$ has the finite subgroup $\Sigma^n$,  
which is a set of coset representatives for $\Gamma_2^n$ in $\Gamma^n$. 
The group $\Sigma^n$ acts on $P^n$ as its group of symmetries. 
The group  $\Sigma^n$ acts transitively on the set of actual vertices of $P^n$ 
and on the set of line edges of $P^n$. 
Hence, every finite subgroup of $\Gamma_2^n$ is conjugate in $\Gamma^n$ 
to a subgroup of either the elementary 2-group $\Kappa^n$ 
or the elementary 2-group $\Lambda^n$ for $2\leq n \leq 7$. 

In a right-angled Coxeter system $(W,S)$ with $S$ finite, two subgroups that are generated 
by subsets of $S$ are conjugate if and only if they are equal, since $S$ projects to a basis 
of the elementary 2-group $W^{ab}$. 
Hence there are $a_n$ conjugacy classes of maximal finite subgroups 
of $\Gamma_2^n$ of order $2^n$, corresponding to the $a_n$ actual vertices of $P^n$, and there 
are $l_n$ conjugacy classes of maximal finite subgroups of $\Gamma_2^n$ of order $2^{n-1}$, corresponding to the $l_n$ line edges of $P^n$ for $2\leq n \leq 7$. 
\end{proof}

A general (but less precise) description of the torsion of $\Gamma^n_2$ is 
given in \cite[Theorem 3]{Ratcliffe99}.  In particular, every finite subgroup of $\Gamma^8_2$ 
is an elementary 2-group, and the maximum order of a finite subgroup of $\Gamma^8_2$ is $2^8$. 

\begin{remark} Corollaries 1, 2, 3 in \cite{Ratcliffe00} and 
Corollary 1 in \cite{Ratcliffe04}, corresponding to the cases $n = 2,\ldots, 5$ of Corollary 2.2,  are incorrect. 
The mistake is that we left out the pointwise stabilizers  of the line edges of $P^n$. 
This mistake does not affect any of the other results in  \cite{Ratcliffe00} and \cite{Ratcliffe04}. 
\end{remark}

The hyperbolic volume of $P^n$, for $2\leq n\leq 8$, is equal to 
$|\Sigma^n|\covol(PO_{n,1}\Z)$, where $\covol(\Gamma)$ is the volume of a fundamental region
for the discontinuous action of $\Gamma$ on $H^n$. 
By \cite{Ratcliffe97,Siegel36}, 
\begin{equation}
  \label{eq:3}
\covol(\po_{n,1}\Z)=\frac{(2^{\frac{n}{2}}\pm 1)\pi^{\frac{n}{2}}}{n!}
\prod_{k=1}^{\frac{n}{2}} |B_{2k}|,
\end{equation}
when $n$ is even, with $B_{2k}$ the $2k$-th Bernoulli number, and the plus sign if 
$n=2,8$ and the minus sign when $n=4, 6$. 
The volume of $P^n$, for $n = 2, 4, 6, 8$, is thus $\pi/2, \pi^2/12, \pi^3/15, 136\pi^4/105$, respectively.  
By Theorem \ref{thm:1} and  \cite[Theorem 17]{Ratcliffe97}, the volume of $P^n$, for $n = 3, 5, 7$, is 
$L(2), 7\zeta(3)/8, 8L(4)$, respectively, where $L(s)$ is the Dirichlet $L$-function
$$L(s) = 1-\frac{1}{3^s}+\frac{1}{5^s}-\frac{1}{7^s}+\cdots $$
and $\zeta(s)$ is the Riemann zeta function. 

By Theorem \ref{thm:1} and the Gauss-Bonnet Theorem, for $n = 2, 4, 6$, we have  
$$\mathrm{vol}(P^n) = \kappa_n\chi(\Gamma^n_2).$$
Hence the Euler characteristic of $\Gamma^n_2$, for $n = 2, 4, 6$, is $-1/4, 1/16, -1/8$, respectively. 
A general formula for $\chi(\Gamma^n_2)$ is given in \cite[Theorem 23]{Ratcliffe97}, 
in particular, $\chi(\Gamma_2^8) = 17/4$.

\section{A Characterization of the Polytopes $P^n$}\label{section:3}

In this section we demonstrate a number of extremal combinatorial properties of the 
$P^n$ polytopes for $2\leq n \leq 6$. 
One spin-off is that manifolds obtained by gluing copies of $P^n$
together are likely candidates to have minimal volume. 
If $P$ is a polytope, let $a_k(P)$ be the number of $k$-dimensional faces of $P$. 

\begin{theorem}\label{thm:2}
For each $n = 2,\ldots, 6$, the polytope $P^n$ has the least number of $k$-faces, for each $k = 1,\ldots, n-1$, 
among all right-angled $n$-dimensional polytopes in $H^n$. 
Moreover for each $n = 3,\ldots, 6$, the polytope $P^n$ is unique up to congruence with this property. 
\end{theorem}
\begin{proof}
We prove the first part by induction on $n$. 
The polygon $P^2$ has the least number of sides among all right-angled polygons of finite area in $H^2$, 
since $P^2$ is a triangle. 
Assume $n > 2$, and $P^{n-1}$  has the least number of $k$-faces, for each $k = 1,\ldots, n-2$, 
among all right-angled $(n-1)$-dimensional polytopes in $H^{n-1}$. 
The polytope $P^n$ has the least number of sides among all right-angled $n$-dimensional polytopes in $H^n$,  
for $n = 3,4$, and $5$ by the inequalities (5), (7), (11)--miss marked as (7)--in Potyagailo-Vinberg \cite{P-V}, and for $n = 6$, 
by Proposition 3 of Dufour \cite{Dufour}. 

Let $P$ be a right-angled $n$-dimensional polytope in $H^n$ 
with sides $S_1,\ldots, S_m$.  
Then each $S_i$ is a right-angled $(n-1)$-dimensional polytope, and $m\geq a_{n-1}(P^n)$. 
Suppose $k$ is in the range $1 \leq k \leq n-2$.  
Each $k$-face of $P$ is a face of exactly $n-k$ sides of $P$, 
since $P$ is right-angled.  Hence we have the main inequality
$$a_k(P) = \frac{1}{n-k}\sum_{i=1}^{m} a_k(S_i) \geq \frac{a_{n-1}(P^n) a_k(P^{n-1})}{n-k} = a_k(P^n).$$
Thus $P^n$ has the least number of $k$-faces, for each $k = 1,\ldots, n-1$, 
among all right-angled $n$-dimensional polytopes in $H^n$. 

Suppose now that $P$ is a right-angled 3-dimensional polytope such that $P$ has the least number of $k$-faces, for each $k = 1, 2$, 
among all right-angled $3$-dimensional polytopes in $H^3$. 
Then $a_2(P) = a_2(P^3) = 6$ and $a_1(P) = a_1(P^3) = 9$. 
From the above main inequality with $k =1$, we deduce that each side of $P$ is a triangle
with angles $0$ or $\pi/2$. Hence, each side of $P$ is either congruent to $P^2$ or is an ideal triangle. 
In particular, $P$ is non-compact.

Let $a_0'(P)>0$ be the number of ideal vertices of $P$. 
As the Euler characteristic of  the closure of $P$ in the projective desk model of $H^3$ is $1$, we have
$$a_0(P)+a_0'(P) - a_1(P) + a_2(P) - a_3(P) = 1,$$ 
and so $a_0(P)+a_0'(P) = 5$. 

The link of an actual vertex of $P$ is a right-angled spherical polygon, hence a triangle, and
the link of an ideal vertex is a right-angled Euclidean polygon, hence a rectangle. 
Thus each actual vertex of $P$ is the endpoint of $3$ edges and each ideal vertex 
is the endpoint of $4$ edges.  Hence 
$$a_1(P) = \big(3a_0(P)+4a_0'(P)\big)/2,$$
and so $a_0(P) = 2$ and $a_0'(P) = 3$ (note that $a_0(P)=6,a_0'(P)=0$ is not possible as $P$
is non-compact). From the equation
$$a_0(P) = \frac{1}{3}\sum_{i=1}^6 a_0(S_i)$$
and the fact that each triangle $S_i$ has at most one actual vertex, 
we deduce that $S_i$ has exactly one actual vertex for each $i$, 
Therefore $S_i$ is congruent to $P^2$ for each $i$. 

To see that $P$ is congruent to $P^3$, position $P$ so that a side of $P$ 
matches with a side of $P^3$ and $P$ lies on the same side of the plane 
spanned by this side as $P^3$ does.  
Then all six sides of $P$ and $P^3$ match up, and so $P = P^3$. 

Suppose now that $n > 3$ and $P^{n-1}$  is unique up to congruence 
with the property of having the least number of $k$-faces, for each $k = 1,\ldots, n-2$, 
among all right-angled $(n-1)$-dimensional polytopes in $H^{n-1}$. 

Let $P$ be a right-angled $n$-dimensional polytope such that $P$ 
has the least number of $k$-faces, for each $k = 1,\ldots, n-1$, 
among all right-angled $n$-dimensional polytopes in $H^n$. 
Then from the above main inequality, 
we deduce that each side $S_i$ of $P$ has the least number of 
$k$-faces, for each $k = 1,\ldots, n-2$, 
among all right-angled $(n-1)$-dimensional polytopes. 
Therefore $S_i$ is congruent to $P^{n-1}$ for each $i$ by the induction hypothesis. 
The same positioning and side-matching argument used in the $n = 3$ case 
implies that $P$ is congruent to $P^n$ for $n = 3, \ldots, 6$. 
\end{proof}

\begin{remark}
The uniqueness part of Theorem \ref{thm:2} is not true for $n = 2$, because 
there are two congruence classes of right-angled triangles  in $H^2$, 
namely the class of $P^2$ and the class of ideal triangles. 
\end{remark}

\begin{theorem}\label{thm:3}
The polygon $P^2$ has the least area among all right-angled polygons in $H^2$. 
\end{theorem}
\begin{proof}
Let $P$ be a right-angled polygon of finite area in $H^2$, and let $\Gamma$ 
be the corresponding right-angled reflection group. 
By the Gauss-Bonnet theorem, 
$$\mathrm{Area}(P) = -2\pi\chi(\Gamma).$$ 
By \cite[Corollary 2]{Chiswell92}, we have
$$\chi(\Gamma) = 1-\frac{a_1(P)}{2}+\frac{a_0(P)}{4} =\frac{4-2a_1(P) +a_0(P)}{4}.$$
Hence the least $-\chi(\Gamma)$ can be is $1/4$, which is attained by $P^2$. 
\end{proof}

\begin{remark}
Both a right-angled quadrilateral with one ideal vertex (which is just $\Delta^2$ reflected across
the side labeled $2$)
and a right-angled pentagon also have the least area among all right-angled polygons in $H^2$.   
\end{remark}

\begin{theorem}\label{thm:4}
The polytope $P^4$ has the least volume among all right-angled 4-dimensional polytopes in $H^4$. 
\end{theorem}
\begin{proof}
Apply the same argument as in the proof of Theorem \ref{thm:3} with
$\mathrm{Volume}(P) = (4\pi^2/3)\chi(\Gamma)$ 
and
\begin{eqnarray*}
\chi(\Gamma) & =  & 1-\frac{a_3(P)}{2}+\frac{a_2(P)}{4} -\frac{a_1(P)}{8} +\frac{a_0(P)}{16} \\ 
                         &  =  & \frac{16 - 8a_3(P)+4a_2(P)-2a_1(P) +a_0(P)}{16}.
\end{eqnarray*}                         
Hence the least $\chi(\Gamma)$ can be is $1/16$, which is attained by $P^4$. 
\end{proof}

\section{The geometry of the Polytope $P^6$}\label{section:4}

We now focus on dimension 6. 
In \S\ref{section:5} we will use some of the remarkable properties of $P^6$ to construct minimal
volume orientable hyperbolic $6$-manifolds, so we record here some facts specific to $P^6$. 
Probably the most beautiful is that P. H. Schoute \cite{Schoute} associated the arrangement of the 27 vertices of the 6-dimensional Gosset
polytope $2_{21}$ with the arrangement of the 27 straight lines in a general cubic surface
(see \cite{Coxeter40}).
Therefore the arrangement of the 27 sides of $P^6$ is related to the arrangement of the 27 
lines in a general cubic surface in a natural way. 
For the history of the 27 lines in a cubic surface, 
see the interesting article of Coxeter \cite{Coxeter83}. 

\begin{figure}
\centering
\begin{pspicture}(0,0)(12.5,7.8)
\rput(6,3.8){\BoxedEPSF{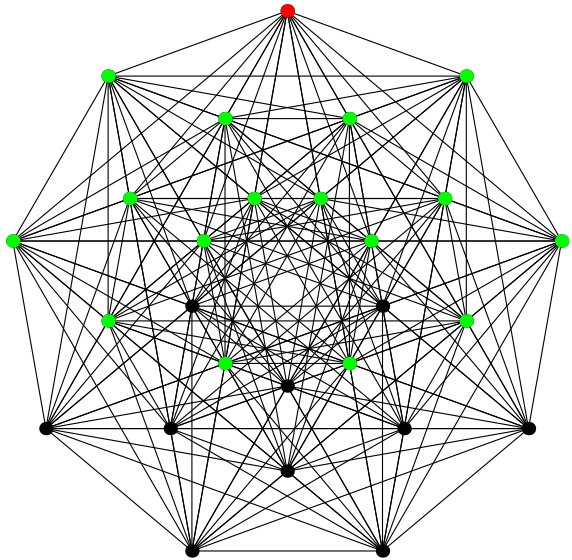 scaled 1050}}
\end{pspicture}
\caption{The 1-skeleton of the 6-dimensional Gosset polytope $2_{21}$. 
The top most vertex (in red) is joined by edges to the green vertices. Some of these edges
lie directly underneath others in this orthogonal projection.}
\label{fig:4}
\end{figure}

The right-angled polytope $P^6$ has 27 sides each congruent to $P^5$.  
The outward Lorentz unit normal vectors for these sides are listed in Table \ref{table:2}.
The vectors $u_1,\ldots, u_6$ in Table \ref{table:2} are normal vectors for the coordinate 
hyperplanes $x_i = 0$, for $i = 1,\ldots, 6$. 
Hence the polytope $P^6$ is positioned in $H^6$ so that six of its sides are bounded 
by the six coordinate hyperplanes $x_i=0$ for $i=1,\ldots, 6$,   
and these six sides intersect at the actual vertex $e_7$ of $P^6$, which is the center of $H^6$. 
The adjacency relation between the 27 sides of $P^6$ is represented by the 1-skeleton of 
the Gosset polytope $2_{21}$ shown in Figure 4. 

The matrix $R_i$ representing the reflection in the $i$th side of $P^6$ with Lorentz unit normal vector $u_i$ 
is given by the standard reflection formula
$$R_ie_j = e_j - 2\langle u_i , e_j\rangle u_i$$
where $e_1,\ldots, e_7$ are the standard basis vectors of $\realnos^7$. 
The reflection in the $i$th side of $P^6$, for $1\leq i \leq 6$, with unit normal vector $u_i = -e_i$,  
is represented by the diagonal $7 \times 7$ matrix $R_i= \mathrm{diag}(1,\ldots, 1, -1, 1,\ldots,1)$ 
with the minus sign in the $i$th slot. 
The reflection in the 7th side of $P^6$, with unit normal vector $u_7 = (1,1,0,0,0,0,1)$, 
 is represented by the Lorentzian $7\times 7$ matrix
$$R_7 = 
\left(\begin{array}{rrrrrrr}
-1 & -2 & 0 & 0 & 0 & 0 & 2 \\
-2 & -1 & 0 & 0 & 0 & 0 & 2 \\
 0 &  0  & 1 & 0 & 0 & 0 & 0 \\
 0 &  0  & 0 & 1 & 0 & 0 & 0 \\
 0 &  0  & 0 & 0 & 1 & 0 & 0 \\
 0 &  0  & 0 & 0 & 0 & 1 & 0 \\
 -2 & -2 & 0 & 0 & 0 & 0 & 3 
 \end{array}\right).
 $$
The unit normal vectors $u_8,\ldots, u_{21}$ are obtained from $u_7$ be permuting the first six coordinates. 
Hence the matrices $R_8,\ldots, R_{21}$ are obtained from $R_7$ be simultaneously permuting the first six rows and columns of $R_7$. 
The reflection in the 22nd side of $P^6$, with unit normal vector $u_{22} = (1,1,1,1,1,0,2)$, 
 is represented by the Lorentzian $7\times 7$ matrix
$$R_{22} = 
\left(\begin{array}{rrrrrrr}
-1 & -2 & -2 & -2 & -2 & 0 & 4 \\
-2 & -1 & -2 & -2 & -2 & 0 & 4 \\
-2 & -2  & -1 & -2 & -2 & 0 & 4 \\
-2 & -2  & -2 & -1 & -2 & 0 & 4 \\
-2 & -2  & -2 & -2 & -1 & 0 & 4 \\
 0 &  0  & 0 & 0 & 0 & 1 & 0 \\
 -4 & -4 & -4 & -4 & -4 & 0 & 9 
 \end{array}\right).
 $$
The unit normal vectors $u_{23},\ldots, u_{27}$ are obtained from $u_{22}$ be permuting the first six coordinates. 
Hence the matrices $R_{23},\ldots, R_{27}$ are obtained from $R_{22}$ be simultaneously permuting the first six rows and columns of $R_{22}$.

The polytope $P^6$ has 72 actual vertices and 27 ideal vertices, 432 ray edges (one actual and one ideal endpoint) 
and 216 line edges (two ideal endpoints),  1080 2-faces each congruent to $P^2$, 720 3-faces each congruent to $P^3$, 
216 ridges each congruent to $P^4$, and 27 sides each congruent to $P^5$.  The face count 72+27, 432+216, 
1080, 720, 216, 27 occurs in the reverse order in Gosset \cite{Gosset} and Coxeter \cite{Coxeter28} 
as the face count for the Euclidean semi-regular polytope $2_{21}$ dual to $P^6$. 
The sides of $2_{21}$ corresponding to the actual vertices of $P^6$ are regular 5-simplices, 
and the sides of $2_{21}$ corresponding to the ideal vertices of $P^6$ are 5-dimensional cross-polytopes
(or $5$-octahedra). 
The link of an actual vertex of $P^6$ is a right-angled spherical 5-simplex, 
and the link of an ideal vertex is a Euclidean 5-dimensional hypercube. 
Table \ref{table:3} lists the 72 actual and 27 ideal vertices of $P^6$.

Finally, as discussed in \S 2, $\vol(P^6)=\pi^3/15$, 
whereas a hyperbolic $6$-manifold $M$ of finite volume with $\chi(M)=-1$
has $\vol(M)=8\pi^3/15$.

\begin{table} 
$$\begin{array}{lll}
u_1 = (-1,0,0,0,0,0,0) &  u_{10} = (1,0,0,1,0,0,1)  &  u_{19} =  (0,0,1,0,0,1,1) \\
u_2 = (0,-1,0,0,0,0,0) &  u_{11} = (0,1,0,1,0,0,1)  &  u_{20} =  (0,0,0,1,0,1,1) \\
u_3 = (0,0,-1,0,0,0,0) &  u_{12} = (0,0,1,1,0,0,1)  &  u_{21} =  (0,0,0,0,1,1,1) \\
u_4 = (0,0,0,-1,0,0,0) &  u_{13} = (1,0,0,0,1,0,1)  &  u_{22} =  (1,1,1,1,1,0,2) \\
u_5 = (0,0,0,0,-1,0,0) &  u_{14} = (0,1,0,0,1,0,1)  &  u_{23} =  (1,1,1,1,0,1, 2) \\
u_6 = (0,0,0,0,0,-1,0) &  u_{15} = (0,0,1,0,1,0,1)  &  u_{24} =  (1,1,1,0,1,1,2) \\
u_7 = (1,1,0,0,0,0,1)  &  u_{16} = (0,0,0,1,1,0,1)  &  u_{25} =  (1,1,0,1,1,1,2) \\
u_8 = (1,0,1,0,0,0,1)  &  u_{17} = (1,0,0,0,0,1,1)  &  u_{26} =  (1,0,1,1,1,1,2) \\
u_9 = (0,1,1,0,0,0,1)  &  u_{18} = (0,1,0,0,0,1,1)  &  u_{27} =  (0,1,1,1,1,1,2) \\
\end{array}$$
\caption{The outward Lorentz unit normal vectors of the sides of $P^6$.}
\label{table:2}
\end{table}
\begin{table}  
$$\begin{array}{lll}
v_1 \  = (0,0,0,0,0,0,1) &  v_{34}= (1,1,1,0,1,2,3)  &  v_{67} = (2,1,2,2,1,1,4) \\
v_2 \  = (0,0,0,1,1,1,2) &  v_{35}= (1,1,1,0,2,1,3)  &  v_{68} = (2,2,1,1,1,2,4) \\
v_3 \  = (0,0,1,0,1,1,2) &  v_{36}= (1,1,1,1,0,2,3)  &  v_{69} = (2,2,1,1,2,1,4) \\
v_4 \  = (0,0,1,1,0,1,2) &  v_{37}= (1,1,1,1,2,0,3)  &  v_{70} = (2,2,1,2,1,1,4) \\
v_5 \  = (0,0,1,1,1,0,2) &  v_{38}= (1,1,1,2,0,1,3)  &  v_{71} = (2,2,2,1,1,1,4) \\
v_6 \  = (0,1,0,0,1,1,2) &  v_{39}= (1,1,1,2,1,0,3)  &  v_{72} = (2,2,2,2,2,2,5) \\
v_7 \  = (0,1,0,1,0,1,2) &  v_{40}= (1,1,1,2,2,2,4)  &  v_{73} = (0,0,0,0,0,1,1) \\
v_8 \  = (0,1,0,1,1,0,2) &  v_{41}= (1,1,2,0,1,1,3)  &  v_{74} = (0,0,0,0,1,0,1) \\
v_9 \  = (0,1,1,0,0,1,2) &  v_{42}= (1,1,2,1,0,1,3)  &  v_{75} = (0,0,0,1,0,0,1) \\
v_{10}= (0,1,1,0,1,0,2) & v_{43}= (1,1,2,1,1,0,3)  &  v_{76} = (0,0,1,0,0,0,1) \\
v_{11}= (0,1,1,1,0,0,2)  & v_{44}= (1,1,2,1,2,2,4) &  v_{77} =  (0,0,1,1,1,1,2) \\
v_{12}= (0,1,1,1,1,2,3)  & v_{45}= (1,1,2,2,1,2,4)  &  v_{78} = (0,1,0,0,0,0,1) \\
v_{13}= (0,1,1,1,2,1,3)  & v_{46}= (1,1,2,2,2,1,4)  &  v_{79} =  (0,1,0,1,1,1,2) \\
v_{14}= (0,1,1,2,1,1,3)  & v_{47}= (1,2,0,1,1,1,3)  &  v_{80} =  (0,1,1,0,1,1,2) \\
v_{15}= (0,1,2,1,1,1,3)  & v_{48}= (1,2,1,0,1,1,3)  &  v_{81} =  (0,1,1,1,0,1,2) \\
v_{16}= (0,2,1,1,1,1,3)  & v_{49}= (1,2,1,1,0,1,3)  &  v_{82} =  (0,1,1,1,1,0,2) \\
v_{17}= (1,0,0,0,1,1,2)  & v_{50}= (1,2,1,1,1,0,3)  &  v_{83} = (1,0,0,0,0,0,1) \\
v_{18}= (1,0,0,1,0,1,2)  & v_{51}= (1,2,1,1,2,2,4)  &  v_{84} =  (1,0,0,1,1,1,2) \\
v_{19}= (1,0,0,1,1,0,2)  & v_{52}= (1,2,1,2,1,2,4)  &  v_{85} =  (1,0,1,0,1,1,2) \\
v_{20}= (1,0,1,0,0,1,2)  & v_{53}= (1,2,1,2,2,1,4)  &  v_{86} =  (1,0,1,1,0,1,2) \\
v_{21}= (1,0,1,0,1,0,2)  & v_{54}= (1,2,2,1,1,2,4)  &  v_{87} =  (1,0,1,1,1,0,2) \\
v_{22}= (1,0,1,1,0,0,2)  & v_{55}= (1,2,2,1,2,1,4)  &  v_{88} =  (1,1,0,0,1,1,2) \\
v_{23}= (1,0,1,1,1,2,3)  & v_{56}= (1,2,2,2,1,1,4)  &  v_{89} =  (1,1,0,1,0,1,2) \\
v_{24}= (1,0,1,1,2,1,3)  & v_{57}= (2,0,1,1,1,1,3)  &  v_{90} =  (1,1,0,1,1,0,2)\\
v_{25}= (1,0,1,2,1,1,3)  & v_{58}= (2,1,0,1,1,1,3)  &  v_{91} =  (1,1,1,0,0,1,2) \\
v_{26}= (1,0,2,1,1,1,3)  & v_{59}= (2,1,1,0,1,1,3)  &  v_{92} =  (1,1,1,0,1,0,2) \\
v_{27}= (1,1,0,0,0,1,2)  & v_{60}= (2,1,1,1,0,1,3)  &  v_{93} =  (1,1,1,1,0,0,2) \\
v_{28}= (1,1,0,0,1,0,2)  & v_{61}= (2,1,1,1,1,0,3)  &  v_{94} =  (1,1,1,1,1,2,3) \\
v_{29}= (1,1,0,1,0,0,2)  & v_{62}= (2,1,1,1,2,2,4)  &  v_{95} =  (1,1,1,1,2,1,3) \\
v_{30}= (1,1,0,1,1,2,3)  & v_{63}= (2,1,1,2,1,2,4)  &  v_{96} =  (1,1,1,2,1,1,3) \\
v_{31}= (1,1,0,1,2,1,3)  & v_{64}= (2,1,1,2,2,1,4)  &  v_{97} =  (1,1,2,1,1,1,3) \\
v_{32}= (1,1,0,2,1,1,3)  & v_{65}= (2,1,2,1,1,2,4)  &  v_{98} =  (1,2,1,1,1,1,3) \\
v_{33}= (1,1,1,0,0,0,2)  & v_{66}= (2,1,2,1,2,1,4)  &  v_{99} =  (2,1,1,1,1,1,3)  \\
\end{array}$$
\caption{The 72 actual and 27 ideal vertices of $P^6$.}
\label{table:3}
\end{table}

\section{The Geometric Construction of the Manifolds} \label{section:5}

We now describe the geometric construction of our examples. 
The volume calculations in \S\ref{section:2} show that
the manifolds we seek will have the right volume if they can be decomposed into
eight copies of $P^6$. 

The set $Q^6=\Kappa^6P^6$, which is the union of 64 copies of $P^6$, 
is a right-angled 6-dimensional polytope with 252 sides. 
In 2003, we constructed hyperbolic 6-manifolds, with $\chi =-8$,
by gluing together the sides of $Q^6$ by a proper side-pairing 
with side-pairing maps of the form $sk$ with $k$ in $\Kappa^6$ and $s$ 
a reflection in a side $S$ of $Q_6$. 
The side-pairing map $sk$ pairs the side $S'=kS$ to $S$.  
For a discussion of proper side-pairings, see \cite[\S 11.1 and \S 11.2]{Ratcliffe06}. 
We call such a side-pairing of $Q^6$ {\it simple}. 
We searched for simple side-pairings of $Q^6$ that yield 
a hyperbolic 6-manifold $M$ with a freely acting $\Z/8$ symmetry group 
that permutes the 64 copies of $P^6$ making up $M$ in such a way 
that the resulting quotient manifold is obtained by gluing 
together 8 copies of $P^6$. 
Such a quotient manifold has $\chi = -8/8=-1$. 
This is easier said than done, since the search space 
of all possible side-pairings of $Q^6$ is very large. 
We succeeded in finding desired side-pairings of $Q^6$ by employing a strategy 
that greatly reduces the search space. 
The strategy is to extend a side-pairing in dimension 5 with the desired 
properties to a side-pairing in dimension 6 with the desired properties. 

Let $Q^5 = \{x\in Q^6: x_1 = 0\}$. 
Then $Q^5$ is a right-angled 5-dimensional polytope with 72 sides. 
Note that $Q^5$ is the union $\Kappa^5P^5$ of 32 copies of $P^5$ where 
$P^5= \{x\in P^6: x_1=0\}$.  
Here we are identifying $P^5$ with side 1 of $P^6$. 
A simple side-pairing of $Q^6$ that yields a hyperbolic 6-manifold $M$ 
restricts to a simple side-pairing of $Q^5$ 
that yields a hyperbolic 5-manifold which is a totally geodesic hypersurface of $M$. 
All the orientable hyperbolic 5-manifolds that are obtained by gluing together 
the sides of $Q^5$ by a simple side-pairing are classified in \cite{Ratcliffe04}. 

We started with the orientable hyperbolic 5-manifold $N$, numbered 27 in \cite{Ratcliffe04},  
obtained by gluing together the sides of $Q^5$ by the simple side-pairing 
with side-pairing code {\tt 2B7JB47JG81}, explained in \cite{Ratcliffe04}. 
The manifold $N$ has a freely acting $\Z/8$ symmetry group 
that permutes the 32 copies of $P^5$ making up $N$ in such a way that 
the resulting quotient manifold is obtained by gluing together 4 copies of $P^5$. 
A generator of the $\Z/8$ symmetry group of $N$ is represented by the Lorentzian $6\times 6$ matrix
$$
\left(\begin{array}{rrrrrr}
    1 & 0 & 0 & 1 & 0 &  -1 \\
    0 & 0 & 0 & 0 & 1 &   0 \\
   -1 & 0 &-1 & 0 & 0 &  1 \\
    0 & 1 & 0 & 0 & 0 &  0 \\
    0 & 0 &-1 &-1 & 0 & 1 \\
   -1 & 0 &-1 & -1 & 0 & 2
        \end{array} \right).
$$
The strategy is to search for simple side-pairings of $Q^6$ that 
extend the side-pairing of $Q^5$ yielding $N$,
and that give hyperbolic 6-manifolds with a freely acting $\Z/8$ 
symmetry group. A generator 
for this $\Z/8$ is represented by the following Lorentzian $7\times 7$ matrix 
$A$ that extends the above $6\times 6$ matrix:
$$
A = \left(\begin{array}{rrrrrrr}
    1 & 0 & 0 & 0 & 0 & 0 & 0 \\
     0 & 1 & 0 & 0 & 1 & 0 & -1 \\
     0 & 0 & 0 & 0 & 0 & 1 & 0 \\
     0 &-1 & 0 &-1 & 0 & 0 & 1 \\
     0 & 0 & 1 & 0 & 0 & 0 & 0 \\
     0 & 0 & 0 &-1 & -1 & 0 & 1 \\
     0 &-1 & 0 &-1 & -1 & 0 & 2
        \end{array} \right).
$$      
For such a side-pairing the resulting quotient 6-manifold can be obtained 
by gluing together 8 copies of $P^6$ by a proper side-pairing. 
By a computer search in 2003, we found 7 simple side-pairings of $Q^6$ 
that induce proper side-pairings of 8 copies of $P^6$ in this way, 
and hence we found 7 hyperbolic $6$-manifolds with $\chi = -1$. 
Each of these 7 manifolds is noncompact with five cusps and
volume $8\vol(P^6)=8\pi^3/15$. 
These 7 hyperbolic $6$-manifolds represent different isometry types, 
since they represent 7 different homology types. 

\begin{table} 
\begin{center}
\begin{tabular}{lllllllll}
$N$&$SP $& $O$ & $C$ & \ \ $H_1$&\ \ $H_2$&\ \ $H_3$&\ \ $H_4$&\ \ $H_5$\\
&&&&$\phantom{\mathbb Z}$0248&$\phantom{\mathbb Z}$0248&$\phantom{\mathbb Z}$0248&
$\phantom{\mathbb Z}$0248&$\phantom{\mathbb Z}$0248\\
1&{\tt MVStfMSJGgJgWDtD2fV84}& 1 & 5 & 
\phantom{${\mathbb Z}$}0401&
\phantom{${\mathbb Z}$}1810&
\phantom{${\mathbb Z}$}4531&
\phantom{${\mathbb Z}$}5000&
\phantom{${\mathbb Z}$}4000\\
2&{\tt MlStfMSJGgJgWDtD2fn84}& 1 & 5 & 
\phantom{${\mathbb Z}$}0301&
\phantom{${\mathbb Z}$}1900&
\phantom{${\mathbb Z}$}3622&
\phantom{${\mathbb Z}$}4000&
\phantom{${\mathbb Z}$}4000\\
3&{\tt k14ONEJdN8ZEdWGYIP1l2}& 0 & 5 & 
\phantom{${\mathbb Z}$}0401&
\phantom{${\mathbb Z}$}1910&
\phantom{${\mathbb Z}$}4821&
\phantom{${\mathbb Z}$}1500&
\phantom{${\mathbb Z}$}0000\\
4&{\tt f65UMFKcN9aEdXHaKU6f3}& 0 & 5 &
\phantom{${\mathbb Z}$}0401&
\phantom{${\mathbb Z}$}1810&
\phantom{${\mathbb Z}$}8710&
\phantom{${\mathbb Z}$}5500&
\phantom{${\mathbb Z}$}0000\\
5&{\tt l15OMFIcN9YEdXHYIO1l3}& 0 & 5 &
\phantom{${\mathbb Z}$}0401&
\phantom{${\mathbb Z}$}2900&
\phantom{${\mathbb Z}$}7810&
\phantom{${\mathbb Z}$}4400&
\phantom{${\mathbb Z}$}1000\\
6&{\tt l65OMFIcN9YEdXHYIO6l3}& 0 & 5 &
\phantom{${\mathbb Z}$}0401&
\phantom{${\mathbb Z}$}2800&
\phantom{${\mathbb Z}$}7910&
\phantom{${\mathbb Z}$}4400&
\phantom{${\mathbb Z}$}1000\\
7&{\tt lx5OMFIcN9YEdXHYIOyl3}& 0 & 5 & 
\phantom{${\mathbb Z}$}0211&
\phantom{${\mathbb Z}$}2800&
\phantom{${\mathbb Z}$}4821&
\phantom{${\mathbb Z}$}1400&
\phantom{${\mathbb Z}$}1000\\
8&{\tt ly5OMFIcN9YEdXHYIOxl3}& 0 & 5 &
\phantom{${\mathbb Z}$}0211&
\phantom{${\mathbb Z}$}2800&
\phantom{${\mathbb Z}$}4930&
\phantom{${\mathbb Z}$}1400&
\phantom{${\mathbb Z}$}1000\\
9&{\tt fx5UMFKcN9aEdXHaKUyf3}& 0 & 5 &
\phantom{${\mathbb Z}$}0301&
\phantom{${\mathbb Z}$}1900&
\phantom{${\mathbb Z}$}5630&
\phantom{${\mathbb Z}$}2500&
\phantom{${\mathbb Z}$}0000\\
\end{tabular}
\end{center}
\caption{Side-pairing codes for $Q^6$, orientability,  number of cusps, 
and homology groups of nine  
hyperbolic 6-manifolds with $\chi = -1$. The side pairing code is explained
in \S\ref{section:6}.}
\label{table:4}
\end{table}

Table \ref{table:4} lists side-pairing codes for simple side-pairings of $Q^6$ 
whose $\Z/8$ quotient manifold has homology groups isomorphic to  
$$\Z^a\oplus(\Z/2)^b\oplus(\Z/4)^c\oplus(\Z/8)^d$$
for nonnegative integers $a,b,c,d$ encoded by $abcd$ in the table. 
The side-pairing code is explained in \S\ref{section:6} below. 
The column headed by $O$  indicates orientability with $1$ indicating orientable. 
The seven nonorientable manifolds in Table \ref{table:4} are the 
manifolds we constructed in 2003 and announced in \cite{E-R-T}.

Note that the side-pairing codes for the nonorientable manifolds in Table \ref{table:4} are permutations 
of the side-pairing codes in \cite[Table 1]{E-R-T} because 
we changed the ordering of the sides of $Q^6$ to conform with our previous 
convention for the ordering of the sides of $Q^5$ in \cite{Ratcliffe04}. 

The seven nonorientable manifolds in Table \ref{table:4}
all have a one-sided, orientable, totally geodesic hypersurface 
isometric to the quotient of $N$ by its freely acting $\integers/8$ symmetry group. 
In \cite{E-R-T}, 
we expressed our wish to construct \emph{orientable\/} hyperbolic 6-manifolds of 
finite volume with $\chi = -1$ 
by a similar construction. This requires the one-sided hypersurface to be nonorientable
by \cite[Theorem III of Chapter X]{S-T}. 

The lynchpin of our construction of nonorientable hyperbolic 6-manifolds with 
$\chi = -1$ is the hyperbolic 5-manifold $N$ 
with its freely acting $\integers/8$ symmetry group. 
We found $N$ by classifying all the orientable hyperbolic 5-manifolds 
that are obtained by gluing together the sides of $Q^5$ by a simple side-pairing. 
In \cite{Ratcliffe04}, we proved that there are 3607 isometry classes of such manifolds.  
In the process of this classification, we determined the group of symmetries of each of these manifolds. 
The manifold $N$ stood out as a manifold with the largest group of symmetries.  

We first envisioned constructing an orientable hyperbolic 6-manifold with $\chi = -1$ 
by a similar construction, using a nonorientable hyperbolic 5-manifold $N'$ 
with a freely acting $\integers/8$ symmetry group. Our first thought was to classify 
all the nonorientable hyperbolic 5-manifolds that are obtained by gluing together the sides of $Q^5$ 
by a simple side-pairing, and look for a freely acting $\integers/8$ symmetry group among 
the most symmetric manifolds.  
We expected a considerably larger number of isometry classes 
than in the orientable case, and so we put off this project in order to wait for Moore's Law to 
increase the power of desktop computers. Indeed, in the last eight years, desktop computers 
have increased in power by a factor of about $2^4$.  When we resumed this project in 2010, 
Steven Tschantz decided to rewrite his original C-programs in the higher level {\it Mathematica} language. 
The {\it Mathematica} programs are less complicated and more robust, 
but they run about 16 times slower than the original C-programs, 
and so we did not gain any more computing power. 

We need a different approach to construct orientable hyperbolic 6-manifolds  with $\chi = -1$. 
Tschantz decided to reverse-engineer our construction by searching directly for 
a proper side-pairing of eight copies of $P^6$ that our original construction produces at its conclusion,  
and finish with a nonorientable hyperbolic 5-manifold $N'$ from which our original construction 
would begin. 

The search space for proper side-pairings of eight copies of $P^6$ is very large. 
The key idea to make the search space reasonable is the new observation that our original construction 
leads to side-pairings of eight copies of $P^6$ of a very restrictive type. 
The epimomorphism $\Gamma^6\rightarrow\Sigma^6$ 
projects the matrix $A$ to the matrix 
$$
\bar{A}=\left(\begin{array}{rrrrrrr}
    1 & 0 & 0 & 0 & 0 & 0 & 0 \\
     0 & -1 & 0 & 0 & -1 & 0 & 1 \\
     0 & 0 & 0 & 0 & 0 & 1 & 0 \\
     0 &-1 & 0 &-1 & 0 & 0 & 1 \\
     0 & 0 & 1 & 0 & 0 & 0 & 0 \\
     0 & 0 & 0 &-1 & -1 & 0 & 1 \\
     0 &-1 & 0 &-1 & -1 & 0 & 2
        \end{array} \right).
$$      
The restriction on the side-pairing of the eight copies of $P^6$ is that sides are paired via 
the action of the cyclic group of order eight generated by $\bar{A}$. 

The matrix $\bar{A}$ acts as a symmetry of $P^6$. 
The permutation $\sigma$ of the 27 sides of $P^6$ induced by $\bar{A}$ is 
the following product of disjoint cycles
$$(2\ 11\ 4\ 20\ 9\ 21\ 12\ 14)(3\ 5\ 18\ 19\ 27\ 15\ 16\ 6)(7\ 17\ 23\ 24 \ 26\ 22\ 13 \ 10)(8\ 25).$$
The ordering of the sides of $P^6$ is as in Table \ref{table:2}.  
In particular, $\bar{A}$ leaves side 1 of $P^6$ invariant.  

Let $8P^6$ be the disjoint union of 8 copies of $P^6$ numbered 1 through 8. 
By a side-pairing of $8P^6$ via the action of the cyclic group generated by $\bar{A}$, 
we mean that we pair side $j$ of polytope $i$ to the side $\sigma^{p_{ij}}(j)$ of polytope $k_{ij}$ 
by the side-pairing transformation $R_{ij}(\bar{A})^{p_{ij}}$ where $R_{ij}$ 
is the reflection in the side $\sigma^{p_{ij}}(j)$ of polytope $k_{ij}$. 
In particular, the point $x$ of side $j$ of polytope $i$ is identified with the point $(\bar{A})^{p_{ij}}(x)$ 
of side $\sigma^{p_{ij}}(j)$ of polytope $k_{ij}$.  
We specify such a side-pairing of $8P^6$ 
by an $8\times 27$ array whose $ij$th term is $k_{ij}^{p_{ij}}$.  

By a computer search, we found two proper side-pairings of  $8P^6$, 
via the action of the cyclic group generated by $\bar{A}$, that yield orientable 
hyperbolic 6-manifolds of finite volume with $\chi = -1$. 
These two side-pairings are given in Table \ref{table:5}. 
We also found proper side-pairings of $8P^6$, 
via the action of the cyclic group generated by $\bar{A}$, 
for the seven nonorientable manifolds in Table \ref{table:4}. 
These seven side-pairings are listed in Table \ref{table:6}. 

From the $8P^6$ side-pairings in Tables \ref{table:5} and \ref{table:6}, 
we derived simple side-pairings of $Q^6$ by positioning 
the first $P^6$ in standard position given by Table \ref{table:2},  
and then developing the $8P^6$ side-pairing to fill out $Q^6$. 
These simple side-pairings of $Q^6$ yield hyperbolic 6-manifolds 
with a freely acting $\integers/8$ symmetry group, 
with generator represented by the matrix $A$, 
whose quotient is isometric to the corresponding $8P^6$ manifold.  
These $Q^6$ side-pairings are given in Table \ref{table:4}. 

\begin{table}[t]  
\begin{center}
\begin{tabular}{l}
1\hskip10pt {\tt MVStfMSJGgJgWDtD2fV84}\\[5pt]
$2^01^72^2\,4^66^23^1\,5^37^78^6\,4^45^15^3\,8^61^16^2\,2^61^12^6\,3^18^04^4\,8^01^75^1\,7^74^62^2$\\
$1^02^71^2\,3^65^24^1\,6^38^77^6\,3^46^16^3\,7^62^15^2\,1^62^11^6\,4^17^03^4\,7^02^76^1\,8^73^61^2$\\
$4^02^21^7\,3^55^35^7\,6^25^57^1\,7^32^46^2\,7^13^35^3\,8^23^38^2\,5^76^07^3\,6^02^22^4\,5^53^51^7$\\
$3^01^22^7\,4^56^36^7\,5^26^58^1\,8^31^45^2\,8^14^36^3\,7^24^37^2\,6^75^08^3\,5^01^21^4\,6^54^52^7$\\
$6^04^63^5\,1^77^12^6\,8^63^35^3\,8^41^58^6\,5^35^57^1\,3^15^53^1\,2^64^08^4\,4^04^61^5\,3^31^73^5$\\
$5^03^64^5\,2^78^11^6\,7^64^36^3\,7^42^57^6\,6^36^58^1\,4^16^54^1\,1^63^07^4\,3^03^62^5\,4^32^74^5$\\
$8^03^54^6\,2^28^68^2\,7^11^16^2\,3^76^47^1\,6^27^78^6\,5^77^75^7\,8^22^03^7\,2^03^56^4\,1^12^24^6$\\
$7^04^53^6\,1^27^67^2\,8^12^15^2\,4^75^48^1\,5^28^77^6\,6^78^76^7\,7^21^04^7\,1^04^55^4\,2^11^23^6$\\[10pt]
2\hskip10pt {\tt MlStfMSJGgJgWDtD2fn84}\\[5pt]
$2^01^72^2\,4^66^23^1\,5^32^48^6\,4^45^15^3\,8^61^16^2\,2^61^12^6\,3^18^04^4\,8^01^75^1\,7^34^62^2$\\
$1^02^71^2\,3^65^24^1\,6^31^47^6\,3^46^16^3\,7^62^15^2\,1^62^11^6\,4^17^03^4\,7^02^76^1\,8^33^61^2$\\
$4^02^21^7\,3^55^35^7\,6^25^17^1\,7^32^46^2\,7^13^35^3\,8^23^38^2\,5^76^07^3\,6^02^22^4\,4^43^51^7$\\
$3^01^22^7\,4^56^36^7\,5^26^18^1\,8^31^45^2\,8^14^36^3\,7^24^37^2\,6^75^08^3\,5^01^21^4\,3^44^52^7$\\
$6^04^63^5\,1^77^12^6\,8^66^45^3\,8^41^58^6\,5^35^57^1\,3^15^53^1\,2^64^08^4\,4^04^61^5\,3^71^73^5$\\
$5^03^64^5\,2^78^11^6\,7^65^46^3\,7^42^57^6\,6^36^58^1\,4^16^54^1\,1^63^07^4\,3^03^62^5\,4^72^74^5$\\
$8^03^54^6\,2^28^68^2\,7^11^56^2\,3^76^47^1\,6^27^78^6\,5^77^75^7\,8^22^03^7\,2^03^56^4\,8^42^24^6$\\
$7^04^53^6\,1^27^67^2\,8^12^55^2\,4^75^48^1\,5^28^77^6\,6^78^76^7\,7^21^04^7\,1^04^55^4\,7^41^23^6$\\[10pt]
\end{tabular}
\end{center}
\caption{Side-pairings of the two orientable $8P^6$ manifolds.}
\label{table:5}
\end{table}

\begin{table}[htp]  
\begin{center}
\begin{tabular}{l}
3\hskip10pt {\tt k14ONEJdN8ZEdWGYIP1l2}\\[5pt]
$2^01^73^0\,5^06^64^2\,3^32^03^0\,8^67^17^3\,4^71^17^1\,5^06^77^3\,1^14^26^6\,5^73^77^6\,2^04^31^7$\\
$1^02^74^0\,6^05^63^2\,4^31^04^0\,7^68^18^3\,3^72^18^1\,6^05^78^3\,2^13^25^6\,6^74^78^6\,1^03^32^7$\\
$4^04^11^0\,6^52^68^2\,6^24^01^0\,2^54^77^0\,7^16^74^7\,6^51^17^0\,6^78^22^6\,2^18^11^5\,4^05^24^1$\\
$3^03^12^0\,5^51^67^2\,5^23^02^0\,1^53^78^0\,8^15^73^7\,5^52^18^0\,5^77^21^6\,1^17^12^5\,3^06^23^1$\\
$6^07^06^5\,1^08^62^2\,2^16^06^5\,6^66^34^1\,3^64^36^3\,1^05^24^1\,4^32^28^6\,6^24^65^6\,6^01^17^0$\\
$5^08^05^5\,2^07^61^2\,1^15^05^5\,5^65^33^1\,4^63^35^3\,2^06^23^1\,3^31^27^6\,5^23^66^6\,5^02^18^0$\\
$8^05^07^3\,1^74^66^2\,8^28^07^3\,3^71^53^0\,7^67^51^5\,1^71^23^0\,7^56^24^6\,2^28^64^7\,8^07^25^0$\\
$7^06^08^3\,2^73^65^2\,7^27^08^3\,4^72^54^0\,8^68^52^5\,2^72^24^0\,8^55^23^6\,1^27^63^7\,7^08^26^0$\\[10pt]
4\hskip10pt {\tt f65UMFKcN9aEdXHaKU6f3}\\[5pt]
$2^01^73^0\,5^06^63^2\,1^23^14^0\,5^37^17^3\,1^61^18^1\,6^07^28^3\,2^14^25^6\,7^21^65^3\,3^11^22^7$\\
$1^02^74^0\,6^05^64^2\,2^24^13^0\,6^38^18^3\,2^62^17^1\,5^08^27^3\,1^13^26^6\,8^22^66^3\,4^12^21^7$\\
$4^03^11^0\,6^51^67^2\,8^51^72^0\,3^63^78^0\,6^66^74^7\,5^53^27^0\,5^78^22^6\,3^26^63^6\,1^78^54^1$\\
$3^04^12^0\,5^52^68^2\,7^52^71^0\,4^64^77^0\,5^65^73^7\,6^54^28^0\,6^77^21^6\,4^25^64^6\,2^77^53^1$\\
$6^07^06^5\,1^08^62^2\,4^28^35^5\,7^15^34^1\,1^54^36^3\,2^07^53^1\,3^31^27^6\,7^51^57^1\,8^34^28^0$\\
$5^08^05^5\,2^07^61^2\,3^27^36^5\,8^16^33^1\,2^53^35^3\,1^08^54^1\,4^32^28^6\,8^52^58^1\,7^33^27^0$\\
$8^05^08^3\,1^73^66^2\,5^76^57^3\,1^61^54^0\,5^37^52^5\,2^74^33^0\,8^55^24^6\,4^35^31^6\,6^55^76^0$\\
$7^06^07^3\,2^74^65^2\,6^75^58^3\,2^62^53^0\,6^38^51^5\,1^73^34^0\,7^56^23^6\,3^36^32^6\,5^56^75^0$\\[10pt]
5\hskip10pt {\tt l15OMFIcN9YEdXHYIO1l3}\\[5pt]
$2^01^73^0\,5^06^63^2\,3^32^04^0\,8^67^17^3\,3^71^18^1\,6^06^78^3\,2^14^25^6\,6^73^78^6\,2^03^32^7$\\
$1^02^74^0\,6^05^64^2\,4^31^03^0\,7^68^18^3\,4^72^17^1\,5^05^77^3\,1^13^26^6\,5^74^77^6\,1^04^31^7$\\
$4^03^11^0\,6^51^67^2\,5^24^02^0\,1^53^78^0\,7^16^74^7\,5^51^17^0\,5^78^22^6\,1^17^11^5\,4^05^24^1$\\
$3^04^12^0\,5^52^68^2\,6^23^01^0\,2^54^77^0\,8^15^73^7\,6^52^18^0\,6^77^21^6\,2^18^12^5\,3^06^23^1$\\
$6^07^06^5\,1^08^62^2\,2^16^05^5\,6^65^34^1\,3^64^36^3\,2^06^23^1\,3^31^27^6\,6^23^66^6\,6^02^18^0$\\
$5^08^05^5\,2^07^61^2\,1^15^06^5\,5^66^33^1\,4^63^35^3\,1^05^24^1\,4^32^28^6\,5^24^65^6\,5^01^17^0$\\
$8^05^08^3\,1^73^66^2\,8^28^07^3\,3^71^54^0\,8^67^52^5\,2^72^23^0\,8^55^24^6\,2^28^63^7\,8^08^26^0$\\
$7^06^07^3\,2^74^65^2\,7^27^08^3\,4^72^53^0\,7^68^51^5\,1^71^24^0\,7^56^23^6\,1^27^64^7\,7^07^25^0$\\[10pt]
6\hskip10pt {\tt l65OMFIcN9YEdXHYIO6l3}\\[5pt]
$2^01^73^0\,5^06^63^2\,3^33^14^0\,8^67^17^3\,3^71^18^1\,6^06^78^3\,2^14^25^6\,6^73^78^6\,3^13^32^7$\\
$1^02^74^0\,6^05^64^2\,4^34^13^0\,7^68^18^3\,4^72^17^1\,5^05^77^3\,1^13^26^6\,5^74^77^6\,4^14^31^7$\\
$4^03^11^0\,6^51^67^2\,5^21^72^0\,1^53^78^0\,7^16^74^7\,5^51^17^0\,5^78^22^6\,1^17^11^5\,1^75^24^1$\\
$3^04^12^0\,5^52^68^2\,6^22^71^0\,2^54^77^0\,8^15^73^7\,6^52^18^0\,6^77^21^6\,2^18^12^5\,2^76^23^1$\\
$6^07^06^5\,1^08^62^2\,2^18^35^5\,6^65^34^1\,3^64^36^3\,2^06^23^1\,3^31^27^6\,6^23^66^6\,8^32^18^0$\\
$5^08^05^5\,2^07^61^2\,1^17^36^5\,5^66^33^1\,4^63^35^3\,1^05^24^1\,4^32^28^6\,5^24^65^6\,7^31^17^0$\\
$8^05^08^3\,1^73^66^2\,8^26^57^3\,3^71^54^0\,8^67^52^5\,2^72^23^0\,8^55^24^6\,2^28^63^7\,6^58^26^0$\\
$7^06^07^3\,2^74^65^2\,7^25^58^3\,4^72^53^0\,7^68^51^5\,1^71^24^0\,7^56^23^6\,1^27^64^7\,5^57^25^0$\\[10pt]
\end{tabular}
\end{center}
\caption{Side-pairings of the seven non-orientable $8P^6$ manifolds.}
\label{table:6}
\end{table}

\begin{table}[htp]  
\begin{center}
\begin{tabular}{l}
7\hskip10pt {\tt lx5OMFIcN9YEdXHYIOyl3}\\[5pt]
$2^01^73^0\,5^06^63^2\,3^33^54^0\,8^67^17^3\,3^71^18^1\,6^06^78^3\,2^14^25^6\,6^73^78^6\,2^43^32^7$\\
$1^02^74^0\,6^05^64^2\,4^34^53^0\,7^68^18^3\,4^72^17^1\,5^05^77^3\,1^13^26^6\,5^74^77^6\,1^44^31^7$\\
$4^03^11^0\,6^51^67^2\,5^24^42^0\,1^53^78^0\,7^16^74^7\,5^51^17^0\,5^78^22^6\,1^17^11^5\,1^35^24^1$\\
$3^04^12^0\,5^52^68^2\,6^23^41^0\,2^54^77^0\,8^15^73^7\,6^52^18^0\,6^77^21^6\,2^18^12^5\,2^36^23^1$\\
$6^07^06^5\,1^08^62^2\,2^16^45^5\,6^65^34^1\,3^64^36^3\,2^06^23^1\,3^31^27^6\,6^23^66^6\,8^72^18^0$\\
$5^08^05^5\,2^07^61^2\,1^15^46^5\,5^66^33^1\,4^63^35^3\,1^05^24^1\,4^32^28^6\,5^24^65^6\,7^71^17^0$\\
$8^05^08^3\,1^73^66^2\,8^26^17^3\,3^71^54^0\,8^67^52^5\,2^72^23^0\,8^55^24^6\,2^28^63^7\,8^48^26^0$\\
$7^06^07^3\,2^74^65^2\,7^25^18^3\,4^72^53^0\,7^68^51^5\,1^71^24^0\,7^56^23^6\,1^27^64^7\,7^47^25^0$\\[10pt]
8\hskip10pt {\tt ly5OMFIcN9YEdXHYIOxl3}\\[5pt]
$2^01^73^0\,5^06^63^2\,3^32^44^0\,8^67^17^3\,3^71^18^1\,6^06^78^3\,2^14^25^6\,6^73^78^6\,3^53^32^7$\\
$1^02^74^0\,6^05^64^2\,4^31^43^0\,7^68^18^3\,4^72^17^1\,5^05^77^3\,1^13^26^6\,5^74^77^6\,4^54^31^7$\\
$4^03^11^0\,6^51^67^2\,5^21^32^0\,1^53^78^0\,7^16^74^7\,5^51^17^0\,5^78^22^6\,1^17^11^5\,4^45^24^1$\\
$3^04^12^0\,5^52^68^2\,6^22^31^0\,2^54^77^0\,8^15^73^7\,6^52^18^0\,6^77^21^6\,2^18^12^5\,3^46^23^1$\\
$6^07^06^5\,1^08^62^2\,2^18^75^5\,6^65^34^1\,3^64^36^3\,2^06^23^1\,3^31^27^6\,6^23^66^6\,6^42^18^0$\\
$5^08^05^5\,2^07^61^2\,1^17^76^5\,5^66^33^1\,4^63^35^3\,1^05^24^1\,4^32^28^6\,5^24^65^6\,5^41^17^0$\\
$8^05^08^3\,1^73^66^2\,8^28^47^3\,3^71^54^0\,8^67^52^5\,2^72^23^0\,8^55^24^6\,2^28^63^7\,6^18^26^0$\\
$7^06^07^3\,2^74^65^2\,7^27^48^3\,4^72^53^0\,7^68^51^5\,1^71^24^0\,7^56^23^6\,1^27^64^7\,5^17^25^0$\\[10pt]
9\hskip10pt {\tt fx5UMFKcN9aEdXHaKUyf3}\\[5pt]
$2^01^73^0\,5^06^63^2\,1^23^54^0\,5^37^17^3\,1^61^18^1\,6^07^28^3\,2^14^25^6\,7^21^65^3\,2^41^22^7$\\
$1^02^74^0\,6^05^64^2\,2^24^53^0\,6^38^18^3\,2^62^17^1\,5^08^27^3\,1^13^26^6\,8^22^66^3\,1^42^21^7$\\
$4^03^11^0\,6^51^67^2\,8^54^42^0\,3^63^78^0\,6^66^74^7\,5^53^27^0\,5^78^22^6\,3^26^63^6\,1^38^54^1$\\
$3^04^12^0\,5^52^68^2\,7^53^41^0\,4^64^77^0\,5^65^73^7\,6^54^28^0\,6^77^21^6\,4^25^64^6\,2^37^53^1$\\
$6^07^06^5\,1^08^62^2\,4^26^45^5\,7^15^34^1\,1^54^36^3\,2^07^53^1\,3^31^27^6\,7^51^57^1\,8^74^28^0$\\
$5^08^05^5\,2^07^61^2\,3^25^46^5\,8^16^33^1\,2^53^35^3\,1^08^54^1\,4^32^28^6\,8^52^58^1\,7^73^27^0$\\
$8^05^08^3\,1^73^66^2\,5^76^17^3\,1^61^54^0\,5^37^52^5\,2^74^33^0\,8^55^24^6\,4^35^31^6\,8^45^76^0$\\
$7^06^07^3\,2^74^65^2\,6^75^18^3\,2^62^53^0\,6^38^51^5\,1^73^34^0\,7^56^23^6\,3^36^32^6\,7^46^75^0$\\[10pt]
\end{tabular}
\end{center}
\begin{center}
{\sc Table 6} (cont.) Side-pairings of the seven non-orientable $8P^6$ manifolds.
\end{center}
\end{table}
\setcounter{table}{6}

\begin{table} 
\begin{center}
\begin{tabular}{cccllllll}
$N$&$CN$ & \ \ $H_1$&\ \ $H_2$&\ \ $H_3$&\ \ $H_4$&\ \ $H_5$\\
&&$\phantom{\mathbb Z}$024&$\phantom{\mathbb Z}$024&$\phantom{\mathbb Z}$024&
$\phantom{\mathbb Z}$024&$\phantom{\mathbb Z}$024\\
1&1&   
\phantom{${\mathbb Z}$}040&
\phantom{${\mathbb Z}$}020&
\phantom{${\mathbb Z}$}040&
\phantom{${\mathbb Z}$}000&
\phantom{${\mathbb Z}$}100\\
1&2&   
\phantom{${\mathbb Z}$}111&
\phantom{${\mathbb Z}$}200&
\phantom{${\mathbb Z}$}211&
\phantom{${\mathbb Z}$}100&
\phantom{${\mathbb Z}$}100\\
1&3&   
\phantom{${\mathbb Z}$}130&
\phantom{${\mathbb Z}$}040&
\phantom{${\mathbb Z}$}030&
\phantom{${\mathbb Z}$}100&
\phantom{${\mathbb Z}$}100\\
1&4&  
\phantom{${\mathbb Z}$}220&
\phantom{${\mathbb Z}$}122&
\phantom{${\mathbb Z}$}120&
\phantom{${\mathbb Z}$}200&
\phantom{${\mathbb Z}$}100\\
1&5&  
\phantom{${\mathbb Z}$}130&
\phantom{${\mathbb Z}$}220&
\phantom{${\mathbb Z}$}230&
\phantom{${\mathbb Z}$}100&
\phantom{${\mathbb Z}$}100\\
2&1&  
\phantom{${\mathbb Z}$}040&
\phantom{${\mathbb Z}$}020&
\phantom{${\mathbb Z}$}040&
\phantom{${\mathbb Z}$}000&
\phantom{${\mathbb Z}$}100\\
2&2&   
\phantom{${\mathbb Z}$}012&
\phantom{${\mathbb Z}$}100&
\phantom{${\mathbb Z}$}112&
\phantom{${\mathbb Z}$}000&
\phantom{${\mathbb Z}$}100\\
2&3&  
\phantom{${\mathbb Z}$}220&
\phantom{${\mathbb Z}$}102&
\phantom{${\mathbb Z}$}120&
\phantom{${\mathbb Z}$}200&
\phantom{${\mathbb Z}$}100\\
2&4&  
\phantom{${\mathbb Z}$}130&
\phantom{${\mathbb Z}$}040&
\phantom{${\mathbb Z}$}030&
\phantom{${\mathbb Z}$}100&
\phantom{${\mathbb Z}$}100\\
2&5&  
\phantom{${\mathbb Z}$}111&
\phantom{${\mathbb Z}$}200&
\phantom{${\mathbb Z}$}211&
\phantom{${\mathbb Z}$}100&
\phantom{${\mathbb Z}$}100\\
3&1&  
\phantom{${\mathbb Z}$}121&
\phantom{${\mathbb Z}$}031&
\phantom{${\mathbb Z}$}120&
\phantom{${\mathbb Z}$}110&
\phantom{${\mathbb Z}$}000\\
3&2&   
\phantom{${\mathbb Z}$}230&
\phantom{${\mathbb Z}$}250&
\phantom{${\mathbb Z}$}230&
\phantom{${\mathbb Z}$}110&
\phantom{${\mathbb Z}$}000\\
3&3&  
\phantom{${\mathbb Z}$}031&
\phantom{${\mathbb Z}$}110&
\phantom{${\mathbb Z}$}230&
\phantom{${\mathbb Z}$}010&
\phantom{${\mathbb Z}$}000\\
3&4&  
\phantom{${\mathbb Z}$}220&
\phantom{${\mathbb Z}$}121&
\phantom{${\mathbb Z}$}030&
\phantom{${\mathbb Z}$}010&
\phantom{${\mathbb Z}$}000\\
3&5&  
\phantom{${\mathbb Z}$}130&
\phantom{${\mathbb Z}$}130&
\phantom{${\mathbb Z}$}130&
\phantom{${\mathbb Z}$}010&
\phantom{${\mathbb Z}$}000\\
4&1&  
\phantom{${\mathbb Z}$}130&
\phantom{${\mathbb Z}$}021&
\phantom{${\mathbb Z}$}120&
\phantom{${\mathbb Z}$}110&
\phantom{${\mathbb Z}$}000\\
4&2&   
\phantom{${\mathbb Z}$}220&
\phantom{${\mathbb Z}$}410&
\phantom{${\mathbb Z}$}600&
\phantom{${\mathbb Z}$}310&
\phantom{${\mathbb Z}$}000\\
4&3&  
\phantom{${\mathbb Z}$}230&
\phantom{${\mathbb Z}$}160&
\phantom{${\mathbb Z}$}040&
\phantom{${\mathbb Z}$}010&
\phantom{${\mathbb Z}$}000\\
4&4&  
\phantom{${\mathbb Z}$}130&
\phantom{${\mathbb Z}$}021&
\phantom{${\mathbb Z}$}120&
\phantom{${\mathbb Z}$}110&
\phantom{${\mathbb Z}$}000\\
4&5&  
\phantom{${\mathbb Z}$}220&
\phantom{${\mathbb Z}$}230&
\phantom{${\mathbb Z}$}220&
\phantom{${\mathbb Z}$}110&
\phantom{${\mathbb Z}$}000\\
5&1&  
\phantom{${\mathbb Z}$}121&
\phantom{${\mathbb Z}$}012&
\phantom{${\mathbb Z}$}120&
\phantom{${\mathbb Z}$}110&
\phantom{${\mathbb Z}$}000\\
5&2&   
\phantom{${\mathbb Z}$}230&
\phantom{${\mathbb Z}$}430&
\phantom{${\mathbb Z}$}610&
\phantom{${\mathbb Z}$}310&
\phantom{${\mathbb Z}$}000\\
5&3&  
\phantom{${\mathbb Z}$}130&
\phantom{${\mathbb Z}$}120&
\phantom{${\mathbb Z}$}130&
\phantom{${\mathbb Z}$}010&
\phantom{${\mathbb Z}$}000\\
5&4&  
\phantom{${\mathbb Z}$}220&
\phantom{${\mathbb Z}$}121&
\phantom{${\mathbb Z}$}030&
\phantom{${\mathbb Z}$}010&
\phantom{${\mathbb Z}$}000\\
5&5&  
\phantom{${\mathbb Z}$}130&
\phantom{${\mathbb Z}$}220&
\phantom{${\mathbb Z}$}230&
\phantom{${\mathbb Z}$}100&
\phantom{${\mathbb Z}$}100\\
6&1&  
\phantom{${\mathbb Z}$}121&
\phantom{${\mathbb Z}$}012&
\phantom{${\mathbb Z}$}120&
\phantom{${\mathbb Z}$}110&
\phantom{${\mathbb Z}$}000\\
6&2&   
\phantom{${\mathbb Z}$}220&
\phantom{${\mathbb Z}$}410&
\phantom{${\mathbb Z}$}600&
\phantom{${\mathbb Z}$}310&
\phantom{${\mathbb Z}$}000\\
6&3&  
\phantom{${\mathbb Z}$}320&
\phantom{${\mathbb Z}$}350&
\phantom{${\mathbb Z}$}140&
\phantom{${\mathbb Z}$}010&
\phantom{${\mathbb Z}$}000\\
6&4&  
\phantom{${\mathbb Z}$}220&
\phantom{${\mathbb Z}$}121&
\phantom{${\mathbb Z}$}030&
\phantom{${\mathbb Z}$}010&
\phantom{${\mathbb Z}$}000\\
6&5&  
\phantom{${\mathbb Z}$}130&
\phantom{${\mathbb Z}$}220&
\phantom{${\mathbb Z}$}230&
\phantom{${\mathbb Z}$}100&
\phantom{${\mathbb Z}$}100\\
7&1&  
\phantom{${\mathbb Z}$}121&
\phantom{${\mathbb Z}$}012&
\phantom{${\mathbb Z}$}120&
\phantom{${\mathbb Z}$}110&
\phantom{${\mathbb Z}$}000\\
7&2&   
\phantom{${\mathbb Z}$}022&
\phantom{${\mathbb Z}$}210&
\phantom{${\mathbb Z}$}420&
\phantom{${\mathbb Z}$}110&
\phantom{${\mathbb Z}$}000\\
7&3&  
\phantom{${\mathbb Z}$}310&
\phantom{${\mathbb Z}$}320&
\phantom{${\mathbb Z}$}130&
\phantom{${\mathbb Z}$}010&
\phantom{${\mathbb Z}$}000\\
7&4&  
\phantom{${\mathbb Z}$}220&
\phantom{${\mathbb Z}$}102&
\phantom{${\mathbb Z}$}030&
\phantom{${\mathbb Z}$}010&
\phantom{${\mathbb Z}$}000\\
7&5&  
\phantom{${\mathbb Z}$}012&
\phantom{${\mathbb Z}$}100&
\phantom{${\mathbb Z}$}112&
\phantom{${\mathbb Z}$}000&
\phantom{${\mathbb Z}$}100\\
8&1&  
\phantom{${\mathbb Z}$}121&
\phantom{${\mathbb Z}$}012&
\phantom{${\mathbb Z}$}120&
\phantom{${\mathbb Z}$}110&
\phantom{${\mathbb Z}$}000\\
8&2&   
\phantom{${\mathbb Z}$}022&
\phantom{${\mathbb Z}$}210&
\phantom{${\mathbb Z}$}420&
\phantom{${\mathbb Z}$}110&
\phantom{${\mathbb Z}$}000\\
8&3&  
\phantom{${\mathbb Z}$}140&
\phantom{${\mathbb Z}$}150&
\phantom{${\mathbb Z}$}140&
\phantom{${\mathbb Z}$}010&
\phantom{${\mathbb Z}$}000\\
8&4&  
\phantom{${\mathbb Z}$}220&
\phantom{${\mathbb Z}$}102&
\phantom{${\mathbb Z}$}030&
\phantom{${\mathbb Z}$}010&
\phantom{${\mathbb Z}$}000\\
8&5&  
\phantom{${\mathbb Z}$}012&
\phantom{${\mathbb Z}$}100&
\phantom{${\mathbb Z}$}112&
\phantom{${\mathbb Z}$}000&
\phantom{${\mathbb Z}$}100\\
9&1&  
\phantom{${\mathbb Z}$}130&
\phantom{${\mathbb Z}$}021&
\phantom{${\mathbb Z}$}120&
\phantom{${\mathbb Z}$}110&
\phantom{${\mathbb Z}$}000\\
9&2&   
\phantom{${\mathbb Z}$}022&
\phantom{${\mathbb Z}$}210&
\phantom{${\mathbb Z}$}420&
\phantom{${\mathbb Z}$}110&
\phantom{${\mathbb Z}$}000\\
9&3&  
\phantom{${\mathbb Z}$}220&
\phantom{${\mathbb Z}$}140&
\phantom{${\mathbb Z}$}030&
\phantom{${\mathbb Z}$}010&
\phantom{${\mathbb Z}$}000\\
9&4&  
\phantom{${\mathbb Z}$}130&
\phantom{${\mathbb Z}$}021&
\phantom{${\mathbb Z}$}120&
\phantom{${\mathbb Z}$}110&
\phantom{${\mathbb Z}$}000\\
9&5&  
\phantom{${\mathbb Z}$}111&
\phantom{${\mathbb Z}$}101&
\phantom{${\mathbb Z}$}111&
\phantom{${\mathbb Z}$}010&
\phantom{${\mathbb Z}$}000
\end{tabular}
\end{center}
\caption{The homology groups of the horospherical cross-sections of the cusps 
of the hyperbolic 6-manifolds in Tables \ref{table:5} and \ref{table:6}.}
\label{table:7}
\end{table}

Both of the orientable $8P^6$ manifolds have a one-sided, nonorientable, totally geodesic hypersurface 
that is isometric to the quotient of a nonorientable hyperbolic 5-manifold $N'$ 
by a freely acting $\integers/8$ symmetry group. The hyperbolic 5-manifold $N'$ is 
obtained by gluing together the sides of $Q^5$ by the simple side-pairing  
that is obtained by restricting the simple side-pairing of $Q^6$ to its $x_1=0$ cross-section. 
The side-pairing code for the simple side-pairing of $Q^5$ for $N'$ is {\tt EKB98LLG6R2}. 

\begin{theorem} 
There exists at least two isometry classes of orientable, noncompact, arithmetic, 
hyperbolic 6-manifolds of finite volume with $\chi = -1$,  
and there exists at least seven isometry classes of nonorientable, noncompact, 
arithmetic, hyperbolic 6-manifolds of finite volume with $\chi = -1$.  
A hyperbolic 6-manifold of finite volume with $\chi = -1$ has least volume among all hyperbolic 6-manifolds. 
\end{theorem}
\begin{proof}
From our construction of the proper side-pairings of $8P^6$ given in 
Tables \ref{table:5} and \ref{table:6}, 
we obtain nine hyperbolic 6-manifolds of finite volume 
by Theorems 11.1.1, 11.1.4, and 11.1.6 in \cite{Ratcliffe06}. 
From the discussion in \S\ref{section:2}, the volume of $8P^6$ is $8\pi^3/15$, 
and so all these manifolds have $\chi = -1$ by the Gauss-Bonnet Theorem
\cite[Theorem 11.3.4]{Ratcliffe06}.  
A hyperbolic 6-manifold of finite volume with $\chi = -1$ has least volume 
among all hyperbolic 6-manifolds by the Gauss-Bonnet Theorem. 

The first two manifolds in Table \ref{table:4} are orientable.  
This can be seen directly from the $8P^6$ side-pairings in Table \ref{table:5}. 
Orient the eight copies of $P^6$ so that the odd numbered polytopes have 
the standard positive orientation and the even numbered polytopes have 
the standard negative orientation.  
This choice of orientation for the eight copies of $P^6$ is compatible with 
the gluing maps, since all the gluing maps reverse the induced orientations 
on the paired sides, and since $\bar{A}$ acts as an orientation reversing 
symmetry of $P^6$.  

The polytope $P^6$ can be truncated along horospherical cross-sections of its cusps 
in such a way that the cross-sections are invariant under the group of symmetries of $P^6$.  
Then our side-pairings of eight copies of $P^6$ restrict to side-pairings of 
eight copies of the truncated $P^6$ that yield 6-manifolds with boundary which 
are strong deformation retracts of our $8P^6$ hyperbolic 6-manifolds. 
The homology groups in Table \ref{table:4} were computed from the cell structures on the 6-manifolds 
with boundary induced by the side-pairings of the eight copies of the truncated $P^6$. 
From our computation of their homology groups given in Table \ref{table:4}, 
we see that the nine hyperbolic 6-manifolds represent nine distinct isometry classes. 

The homology groups of the horospherical cross-sections of the cusps of the manifolds 
are listed in Table \ref{table:7}. The homology groups are isomorphic to 
$$\Z^a\oplus(\Z/2)^b\oplus(\Z/4)^c$$
for nonnegative integers $a,b,c$ encoded by $abc$ in Table \ref{table:7}. 
The last seven manifolds in Table \ref{table:4} are nonorientable, since each has at least 
one nonorientable cusp cross-section indicated by a zero fifth homology group. 

That the first two manifolds in Table \ref{table:4} are orientable and that the last seven 
are nonorientable can be seen directly from the $Q^6$ 
side-pairing codes in Table \ref{table:4}.  This will be explained in 
\S\S \ref{section:6}-\ref{section:7}. 

In \S\ref{section:7}, we will see that each of our $8P^6$ manifolds is isometric to the orbit space 
$H^6/\Gamma$ of a torsion-free subgroup $\Gamma$ of $\Gamma^6$ of finite index. 
Vinberg proved that $\Gamma^6$ is an arithmetic group in \cite{Vinberg67}. 
Hence, all the hyperbolic 6-manifolds we have constructed in this paper are arithmetic. 
\end{proof}

\section{Side-Pairing Coding} \label{section:6}

Table \ref{table:4} lists simple side-pairings for $Q^6$ in a coded form that we now explain. 
The polytope $Q^6$ is defined to be $\Kappa^6P^6$.
The group $\Kappa^6$ is generated by the reflections in the coordinate hyperplanes, 
so $Q^6$ is obtained by reflecting $P^6$ along the coordinate hyperplanes. 
The first 6 sides of $P^6$ lie on these coordinate hyperplanes. 
Thus the sides of $Q^6$ are obtained by applying $\Kappa^6$ to the last 21 sides of $P^6$. 

The elements of $\Kappa^6$ are the diagonal $7\times 7$ matrices $\diag(\pm 1,\ldots, \pm 1, 1)$. 
Hence the outward Lorentz unit normal vectors of the sides of $Q^6$ are 
obtained from the last 21 vectors in Table \ref{table:2} by multiplying all but the last coordinate by $\pm 1$. 
We list the sides in the corresponding standard order, four for each of sides 7 through 21 of $P^6$ 
and 32 for each of sides 22 through 27.  
The outward Lorentz unit normal vectors for the first 4 sides of $Q^6$ are 
$$(1,1,0,0,0,0,1), (-1,1,0,0,0,0,1), (1,-1,0,0,0,0,1), (-1,-1,0,0,0,0,1).$$
Thus $Q^6$ has $4 \cdot 15 + 32 \cdot 6 = 252$ sides, say $S_1, \ldots, S_{252}$. 
The polytope $Q^6$ has sides of two different types, large sides and small sides.  
The large sides are $S_1,\ldots,S_{60}$, each of which is the union of 16 copies of $P^5$ 
joined together along common sides. 
The small sides are $S_{61},\ldots,S_{252}$, each of which is the union of two copies of $P^5$ 
joined together along a common side of each. 

Let $\mathcal Q$ be one of the side-pairings for $Q^6$ listed in Table \ref{table:4}. 
The side-pairing transformation that maps the side $S_i'$ to the side $S_i$ 
is of the form $s_ik_i$ where $s_i$ is the reflection in the side $S_i$ and $k_i$ is an element of $\Kappa^6$. 

We encode an element $\diag(a_1,\ldots, a_6, 1)$ of $\Kappa^6$, where each $a_i = \pm 1$, 
by 
$$\sum_{i=1}^6\frac{(1-a_i)}{2} 2^{i-1}$$
expressed as a single base 64 digit 
$$0,\ldots, 9, A = 10,\ldots, Z =35, a = 36, \ldots, z = 61, @ = 62, \$ = 63.$$
Table \ref{table:8} gives the correspondence between these digits and the elements of $\Kappa^6$.

The side-pairing $\mathcal Q$ is invariant under the group $\Kappa^6$ of symmetries of $Q^6$. 
Hence large sides pair in groups of four so that 
$$k_{4i-3}=k_{4i-2}=k_{4i-1} = k_{4i}$$
for $i=1,\ldots, 15$,  
and small sides pair in groups of 32 so that 
$$k_{32i+29} = k_{32i+30} = \cdots = k_{32i + 60}$$ 
for $i=1,\ldots, 6$. 
Thus $\mathcal Q$ is specified by the 21 digit base 64 code for the sequence
$$k_1,k_5,k_9,\ldots,k_{57},k_{61},k_{93},k_{125},k_{157},k_{189},k_{221}.$$
\begin{table} 
$$
\begin{array}{llllll}
{\tt 0} & \rightarrow & \diag(1,1,1,1,1,1,1) &  {\tt W} & \rightarrow & \diag(1,1,1,1,1,-1,1) \\
{\tt 1} & \rightarrow & \diag(-1,1,1,1,1,1,1) & {\tt X} & \rightarrow &  \diag(-1,1,1,1,1,-1,1) \\
{\tt 2} & \rightarrow & \diag(1,-1,1,1,1,1,1) & {\tt Y} &  \rightarrow & \diag(1,-1,1,1,1,-1,1) \\
{\tt 3} &  \rightarrow & \diag(-1,-1,1,1,1,1,1) & {\tt Z} & \rightarrow & \diag(-1,-1,1,1,1,-1,1) \\
{\tt 4} & \rightarrow & \diag(1,1,-1,1,1,1,1) & {\tt a} & \rightarrow & \diag(1,1,-1,1,1,-1,1) \\
{\tt 5} & \rightarrow & \diag(-1,1,-1,1,1,1,1) & {\tt b} & \rightarrow & \diag(-1,1,-1,1,1,-1,1) \\
{\tt 6} & \rightarrow & \diag(1,-1,-1,1,1,1,1) & {\tt c} & \rightarrow & \diag(1,-1,-1,1,1,-1,1) \\
{\tt 7} & \rightarrow & \diag(-1,-1,-1,1,1,1,1) & {\tt d} & \rightarrow & \diag(-1,-1,-1,1,1,-1,1) \\
{\tt 8} & \rightarrow & \diag(1,1,1,-1,1,1,1) & {\tt e} & \rightarrow & \diag(1,1,1,-1,1,-1,1) \\
{\tt 9} & \rightarrow & \diag(-1,1,1,-1,1,1,1) & {\tt f} & \rightarrow & \diag(-1,1,1,-1,1,-1,1) \\
{\tt A} & \rightarrow & \diag(1,-1,1,-1,1,1,1) & {\tt g} & \rightarrow & \diag(1,-1,1,-1,1,-1,1) \\
{\tt B} & \rightarrow & \diag(-1,-1,1,-1,1,1,1) & {\tt h} & \rightarrow & \diag(-1,-1,1,-1,1,-1,1) \\
{\tt C} & \rightarrow & \diag(1,1,-1,-1,1,1,1) & {\tt i} & \rightarrow & \diag(1,1,-1,-1,1,-1,1) \\
{\tt D} & \rightarrow & \diag(-1,1,-1,-1,1,1,1) & {\tt j} & \rightarrow & \diag(-1,1,-1,-1,1,-1,1) \\
{\tt E} & \rightarrow & \diag(1,-1,-1,-1,1,1,1) & {\tt k} & \rightarrow & \diag(1,-1,-1,-1,1,-1,1) \\
{\tt F} & \rightarrow & \diag(-1,-1,-1,-1,1,1,1) & {\tt l} & \rightarrow & \diag(-1,-1,-1,-1,1,-1,1) \\
{\tt G} & \rightarrow & \diag(1,1,1,1,-1,1,1) & {\tt m} & \rightarrow & \diag(1,1,1,1,-1,-1,1) \\
{\tt H} & \rightarrow & \diag(-1,1,1,1,-1,1,1) & {\tt n} & \rightarrow & \diag(-1,1,1,1,-1,-1,1) \\
{\tt I} & \rightarrow & \diag(1,-1,1,1,-1,1,1) & {\tt o} & \rightarrow & \diag(1,-1,1,1,-1,-1,1) \\
{\tt J} & \rightarrow & \diag(-1,-1,1,1,-1,1,1) & {\tt p} & \rightarrow & \diag(-1,-1,1,1,-1,-1,1) \\
{\tt K} & \rightarrow & \diag(1,1,-1,1,-1,1,1) & {\tt q} & \rightarrow & \diag(1,1,-1,1,-1,-1,1) \\
{\tt L} & \rightarrow & \diag(-1,1,-1,1,-1,1,1) & {\tt r} & \rightarrow & \diag(-1,1,-1,1,-1,-1,1) \\
{\tt M} & \rightarrow & \diag(1,-1,-1,1,-1,1,1) & {\tt s} & \rightarrow & \diag(1,-1,-1,1,-1,-1,1) \\
{\tt N} & \rightarrow & \diag(-1,-1,-1,1,-1,1,1) & {\tt t} & \rightarrow & \diag(-1,-1,-1,1,-1,-1,1) \\
{\tt O} & \rightarrow & \diag(1,1,1,-1,-1,1,1) & {\tt u} & \rightarrow & \diag(1,1,1,-1,-1,-1,1) \\
{\tt P} & \rightarrow & \diag(-1,1,1,-1,-1,1,1) & {\tt v} & \rightarrow & \diag(-1,1,1,-1,-1,-1,1) \\
{\tt Q} & \rightarrow & \diag(1,-1,1,-1,-1,1,1) & {\tt w} & \rightarrow & \diag(1,-1,1,-1,-1,-1,1) \\
{\tt R} & \rightarrow & \diag(-1,-1,1,-1,-1,1,1) & {\tt x} & \rightarrow & \diag(-1,-1,1,-1,-1,-1,1) \\
{\tt S} & \rightarrow & \diag(1,1,-1,-1,-1,1,1) & {\tt y} & \rightarrow & \diag(1,1,-1,-1,-1,-1,1) \\
{\tt T} & \rightarrow & \diag(-1,1,-1,-1,-1,1,1) & {\tt z} & \rightarrow & \diag(-1,1,-1,-1,-1,-1,1) \\
{\tt U} & \rightarrow & \diag(1,-1,-1,-1,-1,1,1) & {\tt @} & \rightarrow & \diag(1,-1,-1,-1,-1,-1,1) \\
{\tt V} & \rightarrow & \diag(-1,-1,-1,-1,-1,1,1) & {\tt \$} & \rightarrow & \diag(-1,-1,-1,-1,-1,-1,1) 
\end{array}
$$
\caption{The decryption of the base 64 digits}
\label{table:8}
\end{table}

A side-pairing transformation $s_ik_i$ preserves orientation if and only if $k_i$ reverses 
orientation.  Therefore the side-pairing transformations of $\mathcal Q$ preserve orientation 
if and only if each digit in its code represents an orientation reversing element of $\Kappa^6$. 
For example, the code for the first $Q^6$ manifold in Table \ref{table:4} is
$$ {\tt MVStfMSJGgJgWDtD2fV84}$$ 
and each digit represents an orientation reversing element of $\Kappa^6$. 
Therefore the hyperbolic 6-manifold $M$ obtained by gluing together the sides of $Q^6$ by 
the above side-pairing is orientable.  
Moreover the corresponding quotient $8P^6$ manifold is also orientable, 
since the matrix $A$ acts as an orientation preserving isometry of $M$, 
since the matrix $A$ preserves orientation.  
Finally, it is worth noting that the hyperbolic 6-manifold $M$ has 27 cusps. 

\section{The Algebraic Interpretation of our Construction} \label{section:7}

The matrix $\bar{A}$ leaves side 1 of $P^6$ invariant, 
and so acts as a symmetry of $P^5$.  
The symmetry group $\Sigma^5$ of $P^5$ 
is a spherical Coxeter group of type $D_5$.  
The matrix $\bar{A}$ corresponds to an element of order 8 in $\Sigma^5$ 
that is conjugate to one, hence all of the Coxeter elements of $\Sigma^5$.  

Let $\widehat{\Gamma}_2^6$ be the subgroup of $\Gamma^6$ generated by $\Gamma_2^6$ and $\bar{A}$, 
and let $C_8 = \langle \bar{A}\rangle$. 
Then we have a split short exact sequence 
$$1\to \Gamma_2^6 \to \widehat\Gamma_2^6 \to C_8\to 1.$$
The group $\Sigma^6$ has a unique conjugacy class of cyclic subgroups of order 8. 
Hence the above extension is unique up to conjugacy in $\Gamma^6$.  

Let ${\mathcal P}$ be one of the $8P^6$ side-pairings in Tables \ref{table:5} or \ref{table:6}. 
Positioning the first $P^6$ in the standard position given by Table \ref{table:2}, 
and then developing ${\mathcal P}$ onto $H^6$ defines an isometry from 
the hyperbolic 6-manifold obtained by gluing together $8P^6$ by ${\mathcal P}$ 
to the orbit space $H^6/\Gamma$ of a discrete torsion-free group $\Gamma$ of isometries of $H^6$. 
From the form of the side-pairing transformations of ${\mathcal P}$, 
we see that $\Gamma$ is a subgroup of $\widehat \Gamma_2^6$ which maps onto $C_8$.  
Hence, we have a short exact sequence 
$$1 \to \Gamma\cap\Gamma_2^6 \to \Gamma \to C_8 \to 1.$$
It is worth noting that the torsion-free subgroup of $\Gamma^6$, with $\chi = -2$, 
constructed in \cite[Theorem 10 (ii)]{E-H} 
also maps onto a $\integers/8$ subgroup of $\Sigma^6$. 

The side-pairing ${\mathcal P}$ was carefully chosen so that $\Gamma$ is generated by 
$\Gamma\cap\Gamma_2^6$ and the matrix $A$. 
The fact that $A$ normalizes $\Gamma\cap\Gamma_2^6$ implies that 
$A$ represents a generator of a freely acting $\integers/8$ group of symmetries of the hyperbolic manifold $H^6/\Gamma\cap\Gamma_2^6$, whose quotient is $H^6/\Gamma$.  
In other words, $H^6/\Gamma\cap\Gamma_2^6$ 
is a regular 8-fold covering of $H^6/\Gamma$ with $A$ representing the generator 
of the $\integers/8$ group of covering transformations of $H^6/\Gamma\cap\Gamma_2^6$. 

The group $\Gamma\cap\Gamma_2^6$ turns out to be a normal subgroup of $\Gamma_2^6$,  
and we have a split short exact sequence
$$1\to \Gamma\cap\Gamma_2^6 \to \Gamma_2^6 \to \Kappa^6 \to 1.$$
Hence, $\Gamma\cap\Gamma_2^6$ is a torsion-free normal subgroup of $\Gamma_2^6$ 
of minimal index 64, and $\Gamma$ is a torsion-free subgroup of $\widehat\Gamma_2^6$ of 
minimal index 64. 

The polytope $Q^6$ is a fundamental polytope for $\Gamma\cap \Gamma_2^6$. 
The $Q^6$ side-pairing code corresponding to $\mathcal{P}$ describes a proper side-pairing 
of the 252 sides of $Q^6$ 
whose side-pairing transformations are in $\Gamma\cap\Gamma_2^6$. 
By Poincar{\'e}'s fundamental polyhedron theorem \cite[Theorem 11.2.2]{Ratcliffe06}, 
the $Q^6$ side-pairing transformations generate $\Gamma\cap \Gamma_2^6$. 
Side-pairing transformations occur in inverse pairs, and so we can choose a set of 126 generators 
for $\Gamma\cap \Gamma_2^6$ among the set of $Q^6$ side-pairing transformations.  
These 126 generators, together with $A$, generate $\Gamma$. 

The matrix $A$ preserves the orientation of $H^6$, since $\det A = 1$.
By inspecting the code for  the $Q^6$ side-pairing for $\Gamma\cap \Gamma_2^6$,  
it is easy to see if the generators of $\Gamma\cap\Gamma_2^6$ preserve orientation. 
Thus it is easy to check  if $\Gamma$ preserves orientation and hence that the
hyperbolic manifold $H^6/\Gamma$ is orientable.

\section{An Algebraic Construction of the first Manifold}\label{section:8}

Let $\mathcal P$ be the first $8P^6$ side-pairing given in Table \ref{table:5}. 
Then $\mathcal P$ determines a torsion-free subgroup $\Gamma$ of $\Gamma^6$ 
with $\chi = -1$; moreover $\Gamma$ is 
generated by the matrix $A$ and the side-pairing transformations of the 
corresponding $Q^6$ side-pairing 
$$\mathcal Q = {\tt MVStfMSJGgJgWDtD2fV84}.$$
The proof that $\Gamma$ is torsion-free relies on the fact that $\mathcal P$ is a proper side-pairing. 
This means, since $P^6$ is right-angled, that the $k$-faces of $8P^6$ are identified in cycles of order $2^{6-k}$ by $\mathcal P$.
Checking these conditions is a labor intensive task, since $8P^6$ has 
576 0-faces (actual vertices), 5,184 1-faces, 8,640 2-faces, 5,760 3-faces, 1,728 4-faces, and 216 5-faces. 
Of course, these conditions were checked by a computer calculation. 

For skeptics, who do not trust computers or computer programmers, 
we give a group theoretical argument that the group $\Gamma$,  
generated by the matrix $A$ and the side-pairing transformations 
of $\mathcal Q$, is a torsion-free subgroup of $\Gamma^6$ with $\chi = -1$. 
Let $\Eta$ be the group generated by the side-pairing transformations of $\mathcal Q$. 

\begin{lemma}\label{lemma:1}  
The group $\Eta$ is a normal subgroup of $\Gamma_2^6$.
\end{lemma}
\begin{proof}
Each side-pairing transformation of $\mathcal Q$ is of the form $sk$ 
with $k$ an element of $\Kappa^6$ and $s$ a reflection in a side of $Q^6$. 
Now $s$ is of the form $\ell r \ell$ with $\ell = \ell^{-1}$ an element of $\Kappa^6$ 
and $r$ a reflection in one of the last 21 sides of $P^6$. 
As $\Kappa^6$ is a subgroup of $\Gamma_2^6$ and all the reflections 
in the sides of $P^6$ are in $\Gamma_2^6$ by Theorem \ref{thm:1}, 
we have that $\Eta$ is a subgroup of $\Gamma_2^6$. 

The group $\Gamma_2^6$ is generated by the reflections in the sides of $P^6$ by Theorem \ref{thm:1}. 
Let $sk$ be a side-pairing transformation of $\mathcal Q$, and let $t$ be a reflection in a side of $P^6$. 
%
First assume that $t\in \Kappa^6$. 
Then $tskt = t\ell r\ell kt = t\ell r \ell t k$, which is a side-pairing transformation of $\mathcal Q$, 
since $\mathcal Q$ is symmetric with respect to $\Kappa^6$. 
Therefore $\Kappa^6$ normalizes $\Eta$. 

Now assume that $t$ is a reflection in one of the last 21 sides of $P^6$.  
Then there is an element $m$ of $\Kappa^6$ such that $tm$ is in $\mathcal Q$. 
As $tskt = tm(mskm)mt$ is in $\Eta$, we see that $t$ normalizes $\Eta$. 
Thus $\Eta$ is a normal subgroup of $\Gamma_2^6$. 
\end{proof}

\begin{lemma}\label{lemma:2}  
The group $\Kappa^6$ is a set of representatives 
for the cosets of $\Eta$ in $\Gamma_2^6$. 
\end{lemma}
\begin{proof} 
The abelianization $(\Gamma_2^6)^{ab}$ of $\Gamma_2^6$ is an elementary 2-group with basis 
the images of the reflections in the 27 sides of $P^6$ by Theorem \ref{thm:1}. 
The group $(\Gamma_2^6/\Eta)^{ab}$ is obtained from $(\Gamma_2^6)^{ab}$ by 
killing the images of the side-pairing transformations of $\mathcal Q$. 
The image of a side-pairing transformation $sk = \ell r \ell k$ is 
the image of $rk$, where $r$ is one of the last 21 sides of $P^6$, and $k$ is an element of $\Kappa^6$ 
depending only on $r$. 
The relation $rk = 1$ allows one to eliminate the generator corresponding to $r$, 
and so $(\Gamma_2^6/\Eta)^{ab}$ is an elementary 2-group of rank 6 
with basis the images of the reflections in the first six sides of $P^6$. 
Hence $\Kappa^6$ projects isomorphically onto  $(\Gamma_2^6/\Eta)^{ab}$. 
Thus any two distinct elements of $\Kappa^6$ lie in different cosets of $\Eta$ in $\Gamma_2^6$. 

Each side-pairing transformation in $\mathcal Q$ is of the form $\ell r\ell k$ 
with $k, \ell \in \Kappa^6$ and $r$ the reflection in any one of the last 21 sides of $P^6$. 
The group $\Gamma_2^6$ is generated by the reflections in the sides of $P^6$ by Theorem \ref{thm:1}, 
and so $\Gamma_2^6$ is generated by $\Eta$ and $\Kappa^6$. 
Therefore $\Eta \Kappa^6 = \Gamma$. 
\end{proof}

\begin{lemma}\label{lemma:3}  
The matrix $A$ normalizes the group $\Eta$, that is, $A\Eta A^{-1} = \Eta$. 
\end{lemma}
\begin{proof}
Let $R_2 = \diag(1,-1,1,1,1,1,1)$.  Then $A = R_2\bar{A}$. 
The reflection matrix $R_2$ normalizes $\Eta$, since $R_2\in \Gamma_2^6$. 
Hence it suffices to show that $\bar{A}$ normalizes $\Eta$. 
The group $\Gamma_2^6/\Eta$ is abelian by Lemma \ref{lemma:2}, 
and so $\Eta$ contains $[\Gamma_2^6,\Gamma_2^6]$. 
The group $(\Gamma_2^6)^{ab} = \Gamma_2^6/[\Gamma_2^6,\Gamma_2^6]$ 
is a $\integers/2$ vector space with basis the images of the reflection matrices $R_1,\ldots, R_{27}$ 
in the sides of $P^6$ by Theorem \ref{thm:1}. 
To show that $\bar{A}$ normalizes $\Eta$, it suffices to show 
that the subspace $\Eta/[\Gamma_2^6,\Gamma_2^6]$ of $(\Gamma_2^6)^{ab}$ 
is invariant under the linear automorphism induced by conjugation by $\bar{A}$. 

The matrix $\bar{A}$ acts as a symmetry of $P^6$, and so induces a permutation $\sigma$ of $\{1,\ldots, 27\}$ 
such that $\bar{A}$ maps side $i$ of $P^6$ to side $\sigma(i)$ of $P^6$ for each $i=1,\ldots, 27$. 
The permutation $\sigma$ is given in \S\ref{section:5}. 
Hence $\bar{A}R_i\bar{A}^{-1} = R_{\sigma(i)}$ for each $i=1,\ldots, 27$. 
We identify $(\Gamma_2^6)^{ab}$ with $(\integers/2)^{27}$ and the image of $R_i$ with 
the $i$th standard basis vector $e_i$ of $(\integers/2)^{27}$. 
Then the linear automorphism of $(\Gamma_2^6)^{ab}$ induced by conjugation by $\bar{A}$
corresponds to the linear automorphism $\sigma_\ast$ of $(\integers/2)^{27}$ 
defined by $\sigma_\ast(e_i) = e_{\sigma(i)}$ for each $i=1,\ldots, 27$. 

The group $\Eta$ is generated by the side-pairing transformations of $\mathcal Q$. 
The images in $(\integers/2)^{27}$ of the side-pairing transformations of $\mathcal Q$ 
are the 21 vectors $v_i = w_i+e_i$, for $i = 7,\ldots, 27$, where $w_i$ is the base 2 vector representing 
the $(i-6)$th digit of the code for $\mathcal Q$.  
Note that the $i$th coordinate of $v_i$ is the only one of the last 21 coordinates of $v_i$ that is nonzero. 
Hence the vectors $v_7,\ldots, v_{27}$ are linearly independent. 
The first digit of $\mathcal Q$ is ${\tt M}$ 
and from Table \ref{table:8}, 
we deduce that $w_7 = e_2+e_3+e_5$.  Thus $v_7 = e_2+e_3+e_5+ e_7$. 
Now we have 
\begin{eqnarray*}
\sigma_\ast(v_7)  & = & \sigma_\ast( e_2+e_3+e_5+e_7)   \\                           
			   & = & e_{\sigma(2)}+e_{\sigma(3)}+e_{\sigma(5)}+e_{\sigma(7)} \\ 
                                & = & e_{11}+e_5+e_{18}+e_{17} \ \
                                = \ \  e_5+e_{11}+e_{17}+e_{18}.
\end{eqnarray*}
The nonzero last 21 coordinates of $\sigma_\ast(v_7)$ tell you how to write $\sigma_\ast(v_7)$ 
as a sum of the vectors $v_i$. 
Observe that 
$$\sigma_\ast(v_7) = v_{11}+v_{17}+v_{18}$$
since 
$$w_{11} + w_{17}+w_{18} = (e_1+e_4+ e_6) + (e_1+ e_2+e_5) + (e_2+e_4+e_6) = e_5. $$
Likewise $\sigma_\ast(v_i)$ can be written as a sum of the vectors $v_j$ for each $i=7,\ldots, 27$. 
Thus the subspace $V$ of $(\integers/2)^{27}$ spanned by the vectors $v_7, \ldots, v_{27}$ is invariant 
under $\sigma_\ast$. 
Table \ref{table:9} gives the matrix for $\sigma_\ast: V \to V$ with respect to the basis $v_7,\ldots, v_{27}$. 
Therefore $\bar{A}$ and $A$ normalize $\Eta$. 
\end{proof}

\begin{table} 
$$\left(
\begin{array}{ccccccccccccccccccccc}
0 & 0 & 0 & 1 & 0 & 0 & 0 & 0 & 0 & 0 & 0 & 0 & 0 & 0 & 0 & 0 & 0 & 0 & 0 & 0 & 0 \\
0 & 0 & 0 & 0 & 0 & 0 & 0 & 0 & 0 & 0 & 0 & 0 & 0 & 0 & 0 & 0 & 0 & 0 & 1 & 0 & 0 \\
0 & 0 & 0 & 0 & 0 & 0 & 0 & 0 & 0 & 0 & 0 & 0 & 0 & 1 & 0 & 0 & 0 & 0 & 0 & 0 & 0 \\
0 & 0 & 0 & 0 & 0 & 0 & 1 & 0 & 0 & 0 & 0 & 0 & 0 & 0 & 0 & 0 & 0 & 0 & 0 & 0 & 0 \\
1 & 1 & 0 & 1 & 0 & 1 & 0 & 1 & 0 & 1 & 1 & 1 & 0 & 0 & 1 & 0 & 1 & 0 & 1 & 0 & 0 \\
0 & 0 & 0 & 0 & 0 & 0 & 0 & 0 & 0 & 0 & 0 & 0 & 0 & 0 & 1 & 0 & 0 & 0 & 0 & 0 & 0 \\
0 & 0 & 0 & 0 & 0 & 0 & 0 & 0 & 0 & 0 & 0 & 0 & 0 & 0 & 0 & 1 & 0 & 0 & 0 & 0 & 0 \\
0 & 0 & 0 & 0 & 0 & 1 & 0 & 0 & 0 & 0 & 0 & 0 & 0 & 0 & 0 & 0 & 0 & 0 & 0 & 0 & 0 \\
0 & 0 & 0 & 0 & 0 & 0 & 0 & 0 & 0 & 0 & 0 & 0 & 0 & 0 & 0 & 0 & 0 & 0 & 0 & 0 & 1 \\
0 & 0 & 0 & 0 & 0 & 0 & 0 & 0 & 1 & 0 & 0 & 0 & 0 & 0 & 0 & 0 & 0 & 0 & 0 & 0 & 0 \\
1 & 0 & 0 & 0 & 0 & 0 & 0 & 0 & 0 & 0 & 0 & 0 & 0 & 0 & 0 & 0 & 0 & 0 & 0 & 0 & 0 \\
1 & 1 & 1 & 1 & 0 & 1 & 1 & 1 & 1 & 0 & 1 & 0 & 0 & 0 & 1 & 0 & 0 & 0 & 1 & 0 & 0 \\
0 & 0 & 0 & 0 & 0 & 0 & 0 & 0 & 0 & 0 & 0 & 1 & 0 & 0 & 0 & 0 & 0 & 0 & 0 & 0 & 0 \\
0 & 1 & 1 & 0 & 1 & 0 & 1 & 0 & 0 & 1 & 0 & 1 & 0 & 1 & 0 & 1 & 0 & 1 & 1 & 1 & 0 \\
0 & 0 & 1 & 0 & 0 & 0 & 0 & 0 & 0 & 0 & 0 & 0 & 0 & 0 & 0 & 0 & 0 & 0 & 0 & 0 & 0 \\
0 & 0 & 0 & 0 & 0 & 0 & 0 & 0 & 0 & 0 & 0 & 0 & 0 & 0 & 0 & 0 & 0 & 0 & 0 & 1 & 0 \\
0 & 0 & 0 & 0 & 0 & 0 & 0 & 0 & 0 & 0 & 1 & 0 & 0 & 0 & 0 & 0 & 0 & 0 & 0 & 0 & 0 \\
0 & 0 & 0 & 0 & 0 & 0 & 0 & 0 & 0 & 0 & 0 & 0 & 0 & 0 & 0 & 0 & 1 & 0 & 0 & 0 & 0 \\
0 & 1 & 0 & 0 & 0 & 0 & 0 & 0 & 0 & 0 & 0 & 0 & 0 & 0 & 0 & 0 & 0 & 0 & 0 & 0 & 0 \\
0 & 0 & 0 & 0 & 0 & 0 & 0 & 0 & 0 & 0 & 0 & 0 & 0 & 0 & 0 & 0 & 0 & 1 & 0 & 0 & 0 \\
0 & 0 & 0 & 0 & 0 & 0 & 0 & 0 & 0 & 0 & 0 & 0 & 1 & 0 & 0 & 0 & 0 & 0 & 0 & 0 & 0
\end{array}
\right)$$
\caption{The matrix for $\sigma_\ast: V \to V$ with respect to the basis $v_7,\ldots, v_{27}$}
\label{table:9}
\end{table}

\begin{lemma} \label{lemma:4}
The group $\Eta$ is torsion-free. 
\end{lemma}
\begin{proof}
Let $C$ be the $6\times 27$ matrix over $\integers/2$ whose $j$th column 
vector represents the image in $\Gamma_2^6/\Eta$ of the reflection $r_j$ in the $j$th side of $P^6$. 
The first six column vectors of $C$ are the standard basis vectors in standard order. 
For $j > 6$, the relation $r_jk_j=1$ implies that the $j$th column vector of $C$ is the base 2 representation of $k_j$
which is represented by the $(j-6)$th digit of the code for $\mathcal Q$. 
The matrix $C$ is given in Table \ref{table:10}. 
 
 By Corollary \ref{corollary:1}, all we have to do to prove that $\Eta$ is torsion-free is to verify 
 that each set of six column vectors of $C$ corresponding to six sides of $P^6$ 
 that intersect in an actual vertex of $P^6$ are linearly independent  
 and that each set of five column vectors of $C$ that correspond to five sides of $P^6$ 
 intersecting in a line edge of $P^6$ are independent. 
 The polytope $P^6$ has 72 actual vertices and 216 line edges,  
  and so there are 288 conditions to check. 
 
 The matrix $\bar{A}$ acts as a symmetry of $P^6$,  
 and $\bar{A}$ normalizes $\Eta$ by the proof of Lemma \ref{lemma:3}.
 Thus, the 288 conditions are symmetric with respect 
 to the action of $\bar{A}$. 
 The matrix $\bar{A}$ acts freely on the sets of actual vertices and line edges of $P^6$, 
 and so we only have to check $288/8 = 36$ conditions,  
 the conditions for the 9 actual vertices with indices 1, 2, 17, 18, 19, 20, 21, 57, 58 in Table \ref{table:3},  
 and the 27 line edges whose ideal endpoints have indices 
 $$
 \begin{array}{l}
 \{73,74\},\{73,75\},\{73,76\},\{73,78\},\{73,79\},\{73,83\}, \{73,84\},\{73,85\},\{73,86\}, \\ 
 \{73,88\},\{73,89\},\{73,91\},\{73,94\},\{76,83\},\{76,85\},\{83,84\},\{83,85\},\{83,86\}, \\
 \{83,87\},\{83,88\},\{83,89\},\{83,92\},\{83,93\},\{83,99\},\{84,85\},\{84,88\},\{84,99\}. 
 \end{array}
 $$
Table \ref{table:11} lists the 36 sets of column indices of  the matrix $C$ 
corresponding to the sides of $P^6$ that intersect in the above 9 actual vertices and 27 line edges of $P^6$.    
We checked that these sets of columns of $C$ are linearly independent  by a computer calculation, 
but these conditions can be checked by hand by a die-hard skeptic. 
 
 For example, the vectors $u_i = -e_i$, for $i=1,\ldots,6$, of Table \ref{table:2} are Lorentz orthogonal to $v_1=e_7$, 
 and so the sides $S_1,\dots, S_6$ of $P^6$ intersect in the actual vertex $v_1$ of $P^6$. 
 The first six column vectors of $C$ are $e_1,\ldots, e_6$, and so they are linearly independent. 
 The vectors $u_i$, for $i=1,\ldots, 4$, and $u_{21} = e_5+e_6+e_7$ 
 of Table \ref{table:2} are Lorentz orthogonal to the ideal endpoints $e_6+e_7$ and $e_5+e_7$ of the line edge $(v_{73}, v_{74})$ of $P^6$, 
 and so the sides $S_1,\ldots, S_4, S_{21}$ of $P^6$ intersect in the line edge $(v_{73}, v_{74})$. 
 The corresponding column vectors, $e_1,\ldots, e_4$, and $e_1+e_2+e_3+e_5+e_6$ of $C$ 
 are obviously linearly independent. 
\end{proof}

\begin{table}  
$$
\begin{array}{cccccccccccccccccccccccc}
\!\!  &  &  &{\tt M}&{\tt V}&{\tt S}&{\tt t}&{\tt f}&{\tt M}&{\tt S}&{\tt J}&{\tt G}&{\tt g}&{\tt J}&{\tt g}&{\tt W}&{\tt D}&{\tt t}&{\tt D}&{\tt 2}&{\tt f}&{\tt V}&{\tt 8}&{\tt 4} \\
\!\!1& &0 & 0      & 1      &  0     &  1    &  1   &  0     &  0      &  1    &  0     &  0     &  1     &  0     &  0      &  1     &  1   &  1      &  0    &  1    &  1    &   0     &  0    \\
\!\!0& &0 &1       & 1      &  0     &  1    &  0   &  1     &  0      &  1    &  0     &  1     &  1     &  1     &  0      &  0     &  1   &  0      &  1    &  0    &  1    &   0     &  0    \\
\!\!0& &0 & 1      & 1      & 1      &  1    &  0   &  1     &  1    &  0     &  0    &  0     &  0     &  0     &  0      &  1     &   1  &  1      &  0    &  0    &  1    &   0   &  1\\
\!\!0& \!\cdots\!&0 & 0     &1       & 1      &  0    &  1   &  0     &  1    &  0     &  0    &  1     &  0     &  1     &  0      &  1    &    0  &  1      &  0   &   1   &   1    &   1   &  0 \\
\!\!0&  &0 & 1      &1       & 1      &  1    &  0   &  1     &  1    &  1     &  1    &  0     &  1     &  0     &  0      &  0    &    1  &  0      &  0   &   0   &   1    &   0   &  0 \\
\!\!0&  &1 & 0      & 0      & 0      &  1    &  1   &  0     &  0    &  0     &  0    &  1     &  0     &  1     &  1      &  0    &    1  &  0      &  0   &   1   &   0    &   0   &  0 
\end{array}
$$
\caption{The $6\times 27$ matrix $C$ whose first six column vectors are $e_1,\ldots, e_6$.}
\label{table:10}
\end{table}

\begin{table} 
$$
\begin{array}{llll}
\{1,2,3,4,5,6\}         & \{1,2,3,4,21\}  & \{3,4,17,18,21\}  & \{2,6,8,10,13\}    \\
\{1,2,3,16,20,21\}  & \{1,2,3,5,20\}  & \{3,5,17,18,20\}  & \{3,4,7,13,17\}   \\
\{2,3,4,13,17,21\}  & \{1,2,4,5,19\}  & \{4,5,17,18,19\}  & \{3,5,7,10,17\}    \\
\{2,3,5,10,17,20\}  & \{1,3,4,5,18\}   & \{17,18,19,20,21\}  & \{4,6,7,8,13\}    \\
\{2,3,6,10,13,16\}  & \{1,3,18,20,21\}  & \{2,4,5,6,8\}  &  \{5,6,7,8,10\}   \\
\{2,4,5,8,17,19\}    & \{2,3,4,5,17\}  &   \{2,4,8,15,19\}  & \{7,8,10,13,17\}\\
\{2,4,6,8,13,15\}    & \{2,3,17,20,21\}  & \{2,3,10,13,17\}  & \{2,13,17,21,26\}    \\
\{2,8,10,13,17,26\}& \{2,4,17,19,21\} & \{2,4,8,13,17\}  & \{3,13,17,21,25\}  \\
\{3,7,10,13,17,25\}& \{2,5,17,19,20\}  & \{2,5,8,10,17\} &  \{10,13,17,25,26\} 
\end{array}
$$
\caption{A list of 36 sets of column indices of the matrix $C$ for which being linearly independent 
 implies that $\Eta$ is torsion-free.}
\label{table:11}
 \end{table}    

\begin{remark}
We first proved that $\Eta$ is torsion-free by showing 
that $\mathcal Q$ is a proper $Q^6$ side-pairing.  
This means, since $Q^6$ is right-angled, that the $k$-faces of $Q^6$ are identified in cycles of order $2^{6-k}$ by $\mathcal Q$.
Checking these conditions is a herculean task (if you are not a computer) since $Q^6$ has 
1344 0-faces (actual vertices), 14,208 1-faces, 23,040 2-faces, 13,920 3-faces, 3,360 4-faces, and 252 5-faces. 
After considering what it takes to prove that $\Eta$ is torsion-free geometrically, 
the task of verifying the 36 conditions in the proof of Lemma \ref{lemma:4} seems quite reasonable.    
\end{remark}

Let $\Gamma$ be the subgroup of $\Gamma^6$ generated by the matrix $A$ and the group $\Eta$. 
Then $\Eta$ is a normal subgroup of $\Gamma$ by Lemma \ref{lemma:3}. 

\begin{lemma}\label{lemma:5}
The quotient group $\Gamma/\Eta$ is cyclic of order 8 generated by $\Eta A$. 
\end{lemma}
\begin{proof}
We have a short exact sequence 
$$1 \to \Gamma_2^6 \to \Gamma^6 \to \Sigma^6 \to 1$$
in which $A$ in $\Gamma^6$ is mapped to $\bar{A}$ in $\Sigma^6$. 
As $\Eta$ is a subgroup of $\Gamma_2^6$, the subgroup $\Gamma$ of $\Gamma^6$ 
is mapped onto the cyclic group $\langle \bar{A}\rangle$ of order 8 
with kernel $\Eta$. 
\end{proof}

\begin{lemma}\label{lemma:6}
The group $\Gamma$ is torsion-free. 
\end{lemma}
\begin{proof}
By Lemmas \ref{lemma:4} and \ref{lemma:5}, the group $\Gamma$ can have only 2-torsion. 
On the contrary, suppose that $\Gamma$ has 2-torsion and let $g$
be an element in $\Gamma$ of order 2. 
By Lemma \ref{lemma:5}, there is an element $h$ of $\Eta$ such that $g = hA^4$. 
Observe that 
\begin{eqnarray*}
A^4 & = & \big(R_2\bar{A}\big)^4 = \big(R_2\bar{A}\big)\big(R_2\bar{A}\big)\big(R_2\bar{A}\big)\big(R_2\bar{A}\big)\\
        & = & R_2\big(\bar{A} R_2 (\bar{A})^{-1}\big)\big((\bar{A})^2R_2(\bar{A})^{-2}\big)\big((\bar{A})^3
        R_2(\bar{A})^{-3}\big)(\bar{A})^4\\
        & = & R_2R_{\sigma(2)}R_{\sigma^2(2)}R_{\sigma^3(2)}(\bar{A})^4
=R_2R_{11}R_4R_{20}(\bar{A})^4.\\
\end{eqnarray*}
Hence we have
\begin{eqnarray*}
1 & = & g^2 = (hA^4)(hA^4)\\
  & = & (hR_2R_{11}R_4R_{20}(\bar{A})^4)(hR_2R_{11}R_4R_{20}(\bar{A})^4) \\
   & = & h(R_2R_{11}R_4R_{20})\big((\bar{A})^4h(\bar{A})^{-4}\big)
        \big((\bar{A})^4R_2R_{11}R_4R_{20}(\bar{A})^{-4}\big) \\
   & = & h(R_2R_{11}R_4R_{20})\big((\bar{A})^4h(\bar{A})^{-4}\big)
        (R_{\sigma^4(2)}R_{\sigma^4(11)}R_{\sigma^4(4)}R_{\sigma^4(20)})\\
   & = & h(R_2R_{11}R_4R_{20})\big((\bar{A})^4h(\bar{A})^{-4}\big)
        (R_9R_{21}R_{12}R_{14}).
\end{eqnarray*}
The elements $h, R_2, R_{11}, R_4, R_{20}, (\bar{A})^4h(\bar{A})^{-4}, R_9, R_{21}, R_{12}, R_{14}$ 
all lie in $\Gamma_2^6$. 
Let $v$ be the image of $h$ in the subspace $V$ of $(\integers/2)^{27}$ 
under the isomorphism from $(\Gamma_2^6)^{ab}$ to $(\integers/2)^{27}$ 
that maps $R_i$ to $e_i$ for each $i=1,\ldots, 27$. 
Then by abelianizing and passing to $(\integers/2)^{27}$, we have 
$$v+ \sigma_\ast^4(v) = e_2 + e_4 + e_9+ e_{11} + e_{12}+ e_{14}+e_{20}+e_{21}.$$
Hence we have 
$$(I+\sigma_\ast^4)(v) = v_9+v_{11}+v_{12}+v_{14}+v_{20}+v_{21}.$$
Table \ref{table:12} gives the matrix for $I+\sigma_\ast^4: V \to V$, with respect to the basis 
$v_7,\ldots, v_{27}$, from which it can be checked 
that the last equation has no solution in $V$. 
Thus no such $g$ exists, and so $\Gamma$ is torsion-free.  
\end{proof}

\begin{table} 
$$
\left(
\begin{array}{ccccccccccccccccccccc}
1 & 0 & 0 & 0 & 0 & 0 & 0 & 0 & 0 & 0 & 0 & 0 & 0 & 0 & 0 & 0 & 0 & 0 & 0 & 1 & 0 \\
0 & 0 & 0 & 0 & 0 & 0 & 0 & 0 & 0 & 0 & 0 & 0 & 0 & 0 & 0 & 0 & 0 & 0 & 0 & 0 & 0 \\
1 & 1 & 1 & 1 & 0 & 1 & 0 & 1 & 0 & 1 & 1 & 1 & 0 & 0 & 1 & 0 & 1 & 0 & 1 & 0 & 0 \\
0 & 0 & 0 & 1 & 0 & 0 & 0 & 0 & 0 & 0 & 0 & 0 & 0 & 0 & 0 & 0 & 0 & 1 & 0 & 0 & 0 \\
0 & 0 & 0 & 0 & 1 & 0 & 0 & 0 & 0 & 0 & 0 & 0 & 0 & 0 & 1 & 0 & 0 & 0 & 0 & 0 & 0 \\
0 & 1 & 1 & 0 & 1 & 1 & 1 & 0 & 0 & 1 & 0 & 1 & 0 & 1 & 0 & 1 & 0 & 1 & 1 & 1 & 0 \\
0 & 0 & 0 & 0 & 0 & 0 & 1 & 0 & 0 & 0 & 0 & 0 & 0 & 0 & 0 & 0 & 1 & 0 & 0 & 0 & 0 \\
0 & 0 & 0 & 0 & 0 & 0 & 0 & 1 & 0 & 0 & 0 & 0 & 0 & 1 & 0 & 0 & 0 & 0 & 0 & 0 & 0 \\
1 & 1 & 1 & 1 & 0 & 1 & 1 & 1 & 0 & 0 & 1 & 0 & 0 & 0 & 1 & 0 & 0 & 0 & 1 & 0 & 0 \\
0 & 0 & 0 & 0 & 0 & 0 & 0 & 0 & 0 & 1 & 0 & 1 & 0 & 0 & 0 & 0 & 0 & 0 & 0 & 0 & 0 \\
0 & 0 & 0 & 0 & 0 & 0 & 0 & 0 & 0 & 0 & 1 & 0 & 0 & 0 & 0 & 1 & 0 & 0 & 0 & 0 & 0 \\
0 & 0 & 0 & 0 & 0 & 0 & 0 & 0 & 0 & 1 & 0 & 1 & 0 & 0 & 0 & 0 & 0 & 0 & 0 & 0 & 0 \\
0 & 0 & 0 & 1 & 1 & 0 & 0 & 0 & 0 & 1 & 0 & 1 & 0 & 0 & 1 & 0 & 0 & 1 & 0 & 0 & 0 \\
0 & 0 & 0 & 0 & 0 & 0 & 0 & 1 & 0 & 0 & 0 & 0 & 0 & 1 & 0 & 0 & 0 & 0 & 0 & 0 & 0 \\
0 & 0 & 0 & 0 & 1 & 0 & 0 & 0 & 0 & 0 & 0 & 0 & 0 & 0 & 1 & 0 & 0 & 0 & 0 & 0 & 0 \\
0 & 0 & 0 & 0 & 0 & 0 & 0 & 0 & 0 & 0 & 1 & 0 & 0 & 0 & 0 & 1 & 0 & 0 & 0 & 0 & 0 \\
0 & 0 & 0 & 0 & 0 & 0 & 1 & 0 & 0 & 0 & 0 & 0 & 0 & 0 & 0 & 0 & 1 & 0 & 0 & 0 & 0 \\
0 & 0 & 0 & 1 & 0 & 0 & 0 & 0 & 0 & 0 & 0 & 0 & 0 & 0 & 0 & 0 & 0 & 1 & 0 & 0 & 0 \\
0 & 0 & 0 & 0 & 0 & 0 & 0 & 0 & 0 & 0 & 0 & 0 & 0 & 0 & 0 & 0 & 0 & 0 & 0 & 0 & 0 \\
1 & 0 & 0 & 0 & 0 & 0 & 0 & 0 & 0 & 0 & 0 & 0 & 0 & 0 & 0 & 0 & 0 & 0 & 0 & 1 & 0 \\
1 & 1 & 1 & 1 & 0 & 1 & 1 & 0 & 0 & 0 & 0 & 0 & 0 & 1 & 1 & 1 & 0 & 0 & 1 & 0 & 0
\end{array}
\right)
$$
\caption{The matrix for $I+\sigma_\ast^4: V\to V$ with respect to the basis $v_7,\ldots, v_{27}$.}
\label{table:12}
\end{table}

\begin{theorem} 
The group $\Gamma$ is a torsion-free subgroup of $\Gamma^6$ with $\chi = -1$. 
\end{theorem}
\begin{proof}
The group $\Gamma$ is torsion-free by Lemma \ref{lemma:6}. 
As discussed in \S\ref{section:2}, we have $\chi(\Gamma^6) = -1/414,720$. 
By \cite[Lemma 16]{Ratcliffe97}--proved algebraically--the 
index of $\Gamma_2^6$ in $\Gamma^6$ is 51,840. 
Hence $\chi(\Gamma_2^6) = -51,840/414,720 = -1/8$. 
The group $\Eta$ has index 64 in $\Gamma_2^6$ by Lemma \ref{lemma:2}, and so $\chi(\Eta) = -64/8 = -8$. 
Finally, $\Eta$ has index 8 in $\Gamma$ by Lemma \ref{lemma:5},  and so $\chi(\Gamma) = -8/8 = -1$. 
\end{proof}

\begin{remark}
The same algebraic argument given in this section works for the $Q^6$ side-pairing codes 
of manifolds 1, 3, 4, 5, 6 in Table \ref{table:4} but fails for the codes of manifolds 2, 7, 8, 9 in Table 
\ref{table:4} because the last equation in the proof of Lemma \ref{lemma:6} does have a solution 
in these four cases. 
\end{remark}

\section{Final Remarks and Open Problems} \label{section:9}

The construction of our minimum volume hyperbolic 6-manifolds would not have been possible 
without the polytope $P^6$, which has the beautiful properties that it is right-angled and all its 27 sides 
are congruent to the right-angled polytope $P^5$.  
These hypercube-like properties make $P^6$ a good building block 
to construct hyperbolic 6-manifolds by gluing together copies of $P^6$ along their sides.   

The group of symmetries of $P^5$ is a spherical Coxeter group of type $D_5$. 
Hence there are $1920$ ways to isometrically glue together two copies of $P^6$ along one side of each. 
This implies that the number of ways to isometrically glue together 8 copies of $P^6$ along sides is extremely large. 
Searching for a proper side-pairing of $8P^6$ is unreasonable 
without restricting the search space of all possible side-pairings 
to a subspace where one is more likely to find a proper side-pairing. 
What is required is more insight into the geometry of $P^6$ that is afforded 
by its relationship to group theory and number theory. 

An amazing property of the hyperbolic polytope $P^6$ is that $P^6$ is a Coxeter polytope 
for the congruence two subgroup $\Gamma_2^6$ of the group $\Gamma^6$ 
of integral, positive, Lorentzian $7\times 7$ matrices, and that $\Gamma^6/\Gamma_2^6$ 
is isomorphic to the group of symmetries of $P^6$. 
This beautiful connection between hyperbolic geometry, 
the theory of Coxeter groups, and number theory gave us 
the necessary insight to find proper side-pairings of $8P^6$.  
Gluing together 8 copies of $P^6$ by a proper side-pairing yields 
a hyperbolic 6-manifold isometric to the orbit space $H^6/\Gamma$ 
of a torsion-free subgroup $\Gamma$ of $\Gamma^6$ of minimal index. 
The way we restricted the search space for proper $8P^6$ side-pairings 
was to search only for proper side-pairings so that $\Gamma\cap\Gamma_2^6$ 
has minimal index in $\Gamma_2^6$, 
which is equivalent to $[\Gamma: \Gamma\cap\Gamma_2^6] = 8$.  
It was this restriction that led us to the symmetry of order 8 of $P^6$, represented by 
the matrix $\bar{A}$, 
that supplies the twists by which we glued 8 copies of $P^6$ together 
to construct our minimum volume hyperbolic 6-manifolds in Tables \ref{table:5} and \ref{table:6}. 

Now to some open problems.

The most basic problem about hyperbolic 6-manifolds is to determine the set of all their possible volumes. 
As the volume of an even dimensional hyperbolic manifold of finite volume is proportional to its Euler characteristic, 
it is enough to solve the following problem. 

\begin{problem}\label{problem:1}
Determine the set of all possible Euler characteristics of hyperbolic 6-manifolds of finite volume.   
\end{problem}

There are also the versions of Problem \ref{problem:1} for the various subclasses of compact, noncompact, orientable, 
or arithmetic hyperbolic 6-manifolds of finite volume.  We know from the first homology group of the first manifold in Table \ref{table:4}
and a covering space argument that there are arithmetic, orientable, noncompact, hyperbolic 6-manifolds of finite volume 
with $\chi = -2^k$, for $k=0,\ldots, 7$. 

\begin{problem}\label{problem:2}
Determine the largest value of the Euler characteristic of a hyperbolic 6-manifold of finite volume 
with an infinite first homology group. 
\end{problem}

If the answer to  Problem \ref{problem:2} is $-1$, then the answer to 
Problem \ref{problem:1} is the set of all negative integers by a covering space  argument.  
Unfortunately, all nine $8P^6$ manifolds in Table \ref{table:4} have a finite first homology group. 

\begin{problem}\label{problem:3}
Classify all the torsion-free subgroups of $\Gamma^6$ of minimal index. 
\end{problem}

The groups corresponding to the nine $8P^6$ manifolds in Table \ref{table:4} are all torsion-free subgroups of $\Gamma^6$ 
of minimal index $414,720$.  
A solution of Problem \ref{problem:3} might produce a group with an infinite first homology group 
and solve Problems \ref{problem:1} and \ref{problem:2}. 

\begin{problem}\label{problem:4}
Determine if $P^6$ has the least volume among all the right-angled 
6-dimensional polytopes in $H^6$.   
\end{problem}

\begin{problem}\label{problem:5}
Determine if the hyperbolic 5-manifold, of volume $7\zeta(3)/4$, 
constructed in \cite{Ratcliffe04} by gluing together two copies of $P^5$ 
has least volume among all noncompact hyperbolic 5-manifolds. 
\end{problem}

The lower bound for the volume of a noncompact hyperbolic 5-manifold 
obtained by Kellerhals \cite{kellerhals98} suggests that $7\zeta(3)/4$ 
is close to the least volume for a noncompact hyperbolic 5-manifold. 

\begin{problem}\label{problem:6}
Determine if $P^7$ has the least number of sides among all right-angled 
7-dimensional polytopes in $H^7$. 
\end{problem}

\begin{problem}\label{problem:7}
Construct a hyperbolic 7-manifold of finite volume 
corresponding to a torsion-free subgroup of $\Gamma^7$ of minimal index.   
\end{problem}

\begin{problem}\label{problem:8}
Construct a hyperbolic 8-manifold of finite volume 
corresponding to a torsion-free subgroup of $\Gamma^8$ of minimal index.   
\end{problem}

\end{document}